\documentclass[11pt,reqno]{amsart}
\usepackage{fullpage}
\usepackage{a4wide}
\usepackage[utf8x]{inputenc}
\usepackage[T1]{fontenc}
\usepackage[english]{babel}
\usepackage{hyperref}
\usepackage{amsfonts,amsmath,amssymb,amsthm,amsrefs}
\usepackage{latexsym}
\usepackage{subfig}
\usepackage{layout}
\usepackage{dsfont}
\usepackage{tikz}
\usepackage{parskip}
\usepackage{graphicx}
\usepackage{caption}
\newtheorem{theorem}{Theorem}
\newtheorem{prop}[theorem]{Proposition}
\newtheorem{lemma}[theorem]{Lemma}

\theoremstyle{definition}

\theoremstyle{remark}
\newtheorem{remark}[theorem]{Remark}

\newcommand{\p}{\mathbb{P}}
\newcommand{\e}{\mathbb{E}}

\newcommand{\reals}{\mathbb{R}}
\newcommand{\ind}{\mathbf{1}}

\newcommand{\me}{\mathrm{e}}
\newcommand{\md}{\mathrm{d}}
\newcommand{\drift}{c}

\def\beq{\begin{eqnarray}} \def\eeq{\end{eqnarray}}

\def\al*#1{\begin{align*}#1\end{align*}}

\def\ga*#1{\begin{gather*}#1\end{gather*}}

\def\alat*#1#2{\begin{alignat*}{#1}#2\end{alignat*}}
\def\bea{\begin{eqnarray*}}
\def\eea{\end{eqnarray*}}
\def\ml*#1{\begin{multline*}#1\end{multline*}}

\allowdisplaybreaks

\begin{document}
\title[]{On the longest/shortest negative excursion of a Lévy risk process and related quantities}

\author{M. A. Lkabous}
\email{M.A.Lkabous@soton.ac.uk}
\address{School of Mathematical Sciences, University of
Southampton, Highfield, Southampton SO17 1BJ, UK}

\author{Z. Palmowski}
\email{zbigniew.palmowski@gmail.com}
\address{Faculty of Pure and Applied Mathematics,
Wroc\l aw University of Science and Technology,
Wyb. Wyspia\'nskiego 27, 50-370 Wroc\l aw, Poland}

\date{\today }

\begin{abstract}
In this paper, we analyze some distributions involving the longest and shortest negative excursions of spectrally negative L\'evy processes using the binomial expansion approach. More specifically, we study the distributions of such excursions and related quantities such as the joint distribution of the shortest and longest negative excursion and their difference (also known as the range) over a random and infinite horizon time. Our results are applied to address new Parisian ruin problems, stochastic ordering and the number near-maximum distress periods showing the superiority  of the binomial expansion approach for such cases.
\end{abstract}

\keywords{Negative excursions, time in the red, L\'evy risk process, Parisian ruin}
\date{\today}
\maketitle

\section{Introduction }

One of the main purposes of ruin theory is to model the surplus of an insurance company with a special focus on the risks inherent to financial distress periods which begin with a capital shortfall. For instance, the duration of such periods (or negative excursions) have been investigated in several papers such as \cite{dosreis1993}, \cite{landriaultetal2020} and \cites{landriaultetal2011}. Monitoring such financial distress periods is also of practical importance since, under the U.S. bankruptcy code, a reorganization is introduced to the firm when it defaults instead of immediate liquidation (see François and Morellec \cite{FM2004}, Galai et al. \cite{galaietal2003} and Li et al. \cite{lietal2014}). In practice, we should be interested in the longest negative duration during a given period as it describes the worst situation an insurance company would experience before recovery. For a risk process  $X$, we denote the duration of the ongoing negative excursion up to time $t$ by
\begin{equation}
U_t=t-g_t, \quad t\geq 0, \label{U}
\end{equation}
where $g_t =\sup \left\{ 0\leq s\leq t : X_s \geq 0 \right\}$ is the last time before $t$ when the process was non-negative, with the convention $\sup \emptyset=0$. This paper continues the analysis on the process $U$ by Landriault et al. \cite{landriault2022bridging} in which last passage times of $U$ have been studied as well as the joint distribution of $( U_{\mathrm{e}_{q}},X_{ \mathrm{e}_{q}})$  where $\mathrm{e}_{q }$ denotes an exponential random variable with rate $q >0$, and by Wang et al. \cite{wang2023refracted} in which a refracted L\'evy risk model with delayed dividend pullbacks is studied.  In this paper, we are interested in the running supremum of $U$ up to time $t$
\begin{equation}
 \bar{U}_t=\sup_{s\leq t} U_s, \quad t\geq 0, \label{supU}
\end{equation}
and in $\bar{U}_\infty=\sup_{s\geq 0} U_s$ being the ultimate supremum. First, we are interested in finding the distribution of $\bar{U}_\infty$, $\bar{U}_{\me_q}$, $\bar{U}_{\tau_b^+}$ and $\bar{U}_{T_b^+}$, where $\mathrm{e}_{q }$ denotes an exponential random variable with rate $q >0$ that is independent of the process $X$, $\tau_b^+$ is the first passage time  and $T_b^+$ is the first passage time under Poisson observations. We use a unified approach based on the binomial expansion approach as in \cite{Yin2014ExactJL}, \cite{li_yin_zhou_2017} and \cite{dosreis1993} among others. Our results are closely related with a set of papers dealing with the so-called red period (see for example \cite{landriaultetal2020}) and Parisian ruin time (see for example \cite{loeffenetal2013}, \cite{guerin_renaud_2015} and \cite{Yin2014ExactJL}). More specifically, we provide a link between our results and the Parisian risk model of Loeffen et al. \cite{loeffenetal2017}. Moreover, we study the distribution of the shortest negative excursion $\underline{U}_{\infty }=\inf_{s\leq t} U_s$. For risk processes with bounded variation, we study new quantities related to the shortest, longest, total negative duration as well as their range, that is the difference between the longest and shortest negative duration. The main contribution of our paper lies in linking the maximal/minimum length of the negative excursions up to some random time with geometric trials also known as the binomial expansion approach. This approach allows to derive some non-trivial identities related to new Parisian ruin models and gives stochastic ordering for the Parisian ruin probabilities completing the existing literature on ordering ruin probabilities in \cite{trufinetal2011} and \cite{lkabous-renaud2018risks}. Finally, making use of the results of Pakes and Yun \cite{li2001number}, we study the number of near-maximum distress periods for the Cram\'{e}r-Lundberg process with exponential jumps.

The rest of the paper is organized as follows. In Section \ref{prelim}, we present the necessary background material on spectrally negative L\'evy processes together with some relevant ruin-related quantities. The main technical analysis is carried out in Section \ref{mainresults} in which we provide the main results of this paper involving formulas for the distribution of the longest and shortest negative duration and the range. Finally, in Section \ref{Parisian} we show how our approach can be applied to solve new Parisian risk models, stochastic ordering of Parisian ruin probabilities and the number of near-maximum negative excursions.

\section{Preliminaries on spectrally negative L\'evy processes}\label{prelim}

First, we present the necessary background material on spectrally negative L\'{e}vy processes. A L\'{e}vy insurance risk process $X$ is a process
with stationary and independent increments and no positive jumps. To avoid
trivialities, we exclude the case where $X$ has monotone paths. Throughout, we will use the standard notation: the law of $X$ when
starting from $X_0 = x$ is denoted by $\mathbb{P}_x$ and the corresponding
expectation by $\mathbb{E}_x$. We write $\mathbb{P}$ and $\mathbb{E}$ when $x=0$. As the L\'{e}%
vy process $X$ has no positive jumps, its Laplace transform is finite: for all $%
\lambda, t \geq 0$,
\begin{equation*}
\mathbb{E} \left[ \mathrm{e}^{\lambda X_t} \right] = \mathrm{e}^{t
	\psi(\lambda)} ,
\end{equation*}
where
\begin{equation*}
\psi(\lambda) = \gamma \lambda + \frac{1}{2} \sigma^2 \lambda^2 +
\int_{(-\infty,0)} \left( \mathrm{e}^{\lambda z} - 1 - \lambda z \mathbf{1}_{\{  z>-1\} } \right) \Pi(\mathrm{d}z) ,
\end{equation*}
for $\gamma \in \mathbb{R}$ and $\sigma \geq 0$, and where $\Pi$ is a $%
\sigma $-finite measure on $(0,\infty)$ called the L\'{e}vy measure of $X$
such that
\begin{equation*}
\int_{(-\infty,0)} (1 \wedge z^2) \Pi(\mathrm{d}z) < \infty .
\end{equation*}

%We now present the definition of the scale functions $W_{q}$ and $Z_{q}$ of $%
%X$. First, recall that there exists a function $\Phifn [0,\infty) \to
%[0,\infty)$ defined
By $\Phi_{q} = \sup \{ \lambda \geq 0 \mid \psi(\lambda)
= q\}$  we denote the right-inverse of $\psi$.
%Hence, it solves
%\begin{equation*}
%\psi ( \Phi_{q} ) = q, \quad q \geq 0 .
%\end{equation*}
%When $\mathbb{E}[X_1]>0$, we have
%\begin{equation}  \label{limesp}
%\lim_{q \rightarrow 0}\dfrac{q}{\Phi_q}=\psi^{\prime }(0+)=\mathbb{E}[X_1].
%\end{equation}
%We have that $\Phi(q)=0$ if and only if $q=0$ and $\psi'(0+)\geq0$.
For $q \geq 0$, the first $q$-scale function of the process $X$ is defined as
the continuous function on $[0,\infty)$ with Laplace transform
\begin{equation}  \label{def_scale}
\int_0^{\infty} \mathrm{e}^{- \lambda y} W_{q} (y) \mathrm{d}y = \frac{1}{%
	\psi(\lambda)-q} , \quad \text{for $\lambda > \Phi_q$.}
\end{equation}
%where $\psi_q(\lambda)=\psi(\lambda) - q$.
%This function is unique, positive
%and strictly increasing for $x\geq0$ and is further continuous for $q\geq0$.
We extend $W_{q}$ to the whole real line by setting $W_{q}(x)=0$ for $x<0$.
We write $W = W_{0}$. We also define the generalized second scale function $Z_{q}(x,\theta )$ by
\begin{equation}  \label{eq:zqscale2}
Z_{q}(x,\theta )=\mathrm{e}^{\theta x}\left( 1-(\psi (\theta )-q) \int_{0}^{x}%
\mathrm{e}^{-\theta y}W_{q}(y)\mathrm{d}y\right) ,\quad x\geq 0,
\end{equation}%
and $Z_{q}(x,\theta )=\mathrm{e}^{\theta x}$ for $x<0$.\\
%For $\theta=0$,
%\begin{equation}  \label{eq:zqscale}
%Z_{q}(x,0)=Z_{q}(x) = 1 + q \int_0^x W_{q}(y)\mathrm{d }y, \quad x \in
%\mathbb{R}.
%\end{equation}
%For $\theta\geq \Phi_q$, using \eqref{def_scale}, the scale function $%
%Z_{q}(x,\theta)$ can be rewritten as
%\begin{equation}  \label{Zv2}
%Z_{q}(x,\theta )=\psi_q (\theta ) \int_{0}^{\infty}\mathrm{e}^{-\theta
%y}W_{q}(x+y)\mathrm{d}y ,\quad x\geq 0.
%\end{equation}
%We recall Kendall's identity that provides the distribution of the first upward crossing of a specific level (see \cite[Corollary VII.3]{bertoin1996}): on $%
%(0,\infty) \times (0,\infty)$, we have
%\begin{equation}  \label{eq:Kendall}
%r \mathbb{P}(\tau_z^+ \in \mathrm{d}r) \mathrm{d}z = z \mathbb{P}(X_r \in
%\mathrm{d}z) \mathrm{d}r .
%\end{equation}
%Note that we can show
%\begin{equation}  \label{L11}
%\Lambda ^{\left(q\right)}\left(0,r\right)= \mathrm{e}^{qr}.
%\end{equation}
Throughout this paper, we assume that the first scale function $W_q(x)$ is
continuously differentiable. %For this to hold it is enough to assume that the distribution function $F$ is absolutely continuous.
%the following (regularity) condition is satisfied:
%\begin{equation}\label{eq:condW}
%\sigma>0\quad\text{or}\quad\int_{-1}^0 x \Pi(\mathrm{d} x) = \infty \quad\text{or}\quad \Pi(\mathrm{d} x
%)<<{\mathrm d} x.
%\end{equation}

%We also denote the partial derivative of $\Lambda ^{\left( q\right) }$ with
%respect to $x$ by
%\begin{equation}\label{DS2}
%\Lambda ^{(q)^{\prime }}(x,z)=\frac{\partial \Lambda ^{(q)}}{\partial x}%
%(x,z)=\int_{0}^{\infty }W_{q}^{\prime }\left( x+u\right) \frac{u}{z}\mathbb{P%
%}\left( X_{z}\in \mathrm{d}u\right).
%\end{equation}
%Note that, by
%\eqref{eq:condW} and
%Leibnitz rule, $\Lambda ^{(q)^{\prime }}(x,r)$ is well-defined since
%$W_q$ is continuously differentiable and above integral is finite.
\subsection{Fluctuation identities} 
We define of the first-passage time of $X$ above a level $b \in \reals$ as
\begin{align*}
\tau_b^+ =
\inf\{t>0 \colon X_t > b\} ,
\end{align*}
with the convention $\inf \emptyset=\infty$. For $q \geq 0$ and $x \leq b$ , we have
\begin{equation}\label{Laptrans}
\mathbb{E}_{x} \left[ \mathrm{e}^{-q \tau_b^+} \mathbf{1}_{\{ \tau_b^{+} < \infty \}} \right] = \me^{-\Phi_q (b-x)} .
\end{equation}
From Loeffen et al. \cite{loeffenetal2017}, we recall the following function $\Lambda ^{\left( q\right) } (x,r) $
\begin{equation*}
\Lambda ^{\left( q\right) }\left( x,r\right) =\int_{0}^{\infty }W_{q}\left(x+u\right) \frac{u}{r}\mathbb{P}\left( X_{r}\in \mathrm{d}u\right) ,
\end{equation*}
and we write $\Lambda =\Lambda ^{(0)}$ when $q=0$. From Loeffen et al. \cite{loeffenetal2013}, for $x\leq 0$, we have
\begin{equation}\label{nowy}
\mathbb{P}_x(\tau_0^+<r)=\Lambda\left( x,r\right).
\end{equation}
The Laplace transform of the \textit{classical} ruin time $\tau_0^- =\inf\{t>0 \colon X_t <0\} $ is given by
\begin{equation}  \label{E:classicalruinprobaX1}
\e_x \left[\me^{-q \tau_0^-} \ind_{\left\lbrace \tau_0^- < \infty \right\rbrace }\right] = Z_q (x)-\dfrac{q}{\Phi_q} W_q (x),
\end{equation}
for $q>0$. %Moreover, the \textit{classical} probability of ruin is given by
%\begin{equation}  \label{E:classicalruinprobaX}
%\p_x \left( \tau_0^- < \infty \right) = 1 - \e \left[ X_1 \right]W(x).
%\end{equation}
In particular, for $\e \left[ X_1 \right]>0$, the \textit{classical} probability of ruin is given by
\begin{equation}  \label{E:classicalruinprobaX}
\p_x \left( \tau_0^- < \infty \right) = 1 - \e \left[ X_1 \right]W(x).
\end{equation}
From \cite[Lemma 2.2]{loeffenetal2014} we know that for $p \geq 0$ and $z\in \reals$  we have
\begin{equation}\label{id1}
\mathbb{E}_{x}\left[ \mathrm{e}^{-q \tau
_{0}^{-}}W_{q }\left( X_{\tau _{0}^{-}}+z\right) \mathbf{1}_{\left\{
\tau _{0}^{-}<\infty \right\} }\right] =W_{q }\left( x+z\right) -W_{q}(x)\me^{\Phi _{q}z}.
\end{equation}
\subsection{Cram\'{e}r-Lundberg process with exponential jumps}
Our results will be illustrated by numerical examples for the special case of Cramér-Lundberg risk process $X = \{X_t, t\ge 0\}$ with exponential claims, that is
\begin{equation}\label{CL}
X_t = X_0+ \drift t - A_t =X_0+ \drift t - \sum_{i=1}^{N_t} C_i ,
\end{equation}
where $c>0$ is the fixed premium rate, $N=\{N_t, t\ge 0\}$ is a Poisson process with rate $\eta>0$ and $\{C_i\}_{i \in \mathbb{N}^+}$ is a sequence of i.i.d. random variables. The Poisson process is independent of $\{C_i\}_{i\in \mathbb{N}^+}$.
%Gerber \cite{gerber1990does} derived the following results involving $ \mathcal{S}_%\infty$,%
%\begin{eqnarray*}
%\mathbb{E}\left[ \mathrm{e}^{\Phi _{p} \mathcal{S}_\infty }\right]
%&=&\frac{\left( 1-\lambda \mathbb{E}\left[ X_{1}\right] \right) }{\left(
%1-\lambda m^{\prime }\left( p\right) \right) ^{2}}, \\
%\mathbb{E}\left[ \mathcal{S}_\infty  \right]  &=&\frac{\lambda \mathbb{E}\left[
%X_{1}^{2}\right] }{\left( 1-\lambda \mathbb{E}\left[ X_{1}\right] \right)
%^{2}}, \\\mathbb{V}\left[ \mathcal{S}_\infty \right]  &=&\frac{\lambda \mathbb{E}\left[
%X_{1}^{3}\right] }{\left( 1-\lambda \mathbb{E}\left[ X_{1}\right] \right)
%^{3}}+\frac{2\left( \lambda \mathbb{E}\left[ X_{1}^{2}\right] \right) ^{2}}{%
%\left( 1-\lambda \mathbb{E}\left[ X_{1}\right] \right) ^{4}}.
%\end{eqnarray*}
We assume that $\{C_{1},C_{2},\dots \}$ are exponentially distributed
random variables with a parameter $\alpha $ and that the net-profit condition holds, that is, that
\[c-\eta /\alpha >0.\]
Then the scale function of $X$ is known to be
\begin{align}\label{scalf}
W(x)&=\frac{1}{c-\eta /\alpha }\left( 1-\frac{\eta }{c\alpha }\mathrm{e}^{(
\frac{\eta }{c}-\alpha )x}\right).
\end{align}
One can observe that a sum of i.i.d. exponential random variables has the same law as a gamma random variable. Since $C_i$ and $N_t$ are independent, from eq. $(17)$ in \cite{gomezetal2011} we have
\begin{eqnarray*}
\mathbb{P} \left( \sum_{i=1}^{N_r}C_i\in\mathrm{d}z \right)= \mathrm{e}^{-\eta r} \left( \delta_0(\mathrm{d}z) + \me^{-\alpha z}\sqrt{\dfrac{r\eta \alpha}{z}}I_1 (2\sqrt{r\eta \alpha z}) \mathrm{d}z  \right) ,
\end{eqnarray*}
%\begin{multline*}
%\mathbb{P} \left( \sum_{i=1}^{N_r}C_i\in\mathrm{d}y \right) = \sum_{k=0}^\infty \mathbb P \left( \sum_{i=0}^k C_i\in\mathrm{d}y \right) \mathbb P(N_r=k)  \\
%= \mathrm{e}^{-\eta r} \left( \delta_0(\mathrm{d}y)  + \mathrm{e}^{-\alpha y}\sum_{m=0}^\infty \frac{ (\alpha \eta r)^{m+1}}{m!(m+1)!}  y^{m} \mathrm{d}y  \right) ,
%\end{multline*}
where $I_{1}$ is the modified Bessel function of the first kind of order $1$ and  $\delta_0(\mathrm{d}y)$ is a Dirac mass at $0$. Consequently,
the law of $X$ is given by
\begin{align*}
\p\left( X_r \in \md z\right) = \me^{-\eta r}\left( \delta_0 (cr-\md z) + \me^{-\alpha (cr-z)}\sqrt{\dfrac{r\eta \alpha}{cr-z}}I_1 (2\sqrt{r\eta \alpha (cr-z)}) \mathrm{d}z\right) , \quad z \le cr
\end{align*}
and hence
\begin{equation*}
\int_{0}^{\infty }\frac{z}{r}\mathbb{P}\left( X_{r}\in \mathrm{d}z\right) =%
\mathrm{e}^{-\eta r}\left( c+\int_{0}^{cr}z\mathrm{e}^{-\alpha (cr-z)}\sqrt{%
\dfrac{\eta \alpha }{\left( cr-z\right) r}}I_{1}(2\sqrt{r\eta \alpha (cr-z)})%
\mathrm{d}z\right).
\end{equation*}
Further, using the same approach as in \cite[p. 9]{loeffenetal2013}, we derive
\begin{equation}\label{int1}
 \frac{\eta}{c\alpha}\int_0^\infty \mathrm{e}^{(\frac{\eta}{c}-\alpha)z}z\mathbb{P}(X_r \in \mathrm{d}z) = \int_0^\infty z \mathbb{P}(X_r\in\mathrm{d}z)- (c-\eta/\alpha )r.
\end{equation}
Thus,
\begin{eqnarray*}
\Lambda \left( x,r\right)  &=&\frac{1}{c-\eta /\alpha }\int_{0}^{\infty
}\left( 1-\frac{\eta }{c\alpha }e^{\left( \eta /c-\alpha \right) \left(
x+z\right) }\right) \frac{z}{r}\mathbb{P}\left( X_{r}\in dz\right)  \\
&=&\frac{1}{c-\eta /\alpha }\int_{0}^{\infty }\frac{z}{r}\mathbb{P}\left(
X_{r}\in dz\right) -\frac{\eta e^{\left( \eta /c-\alpha \right) x}}{c\alpha
\left( c-\eta /\alpha \right) }\int_{0}^{\infty }e^{\left( \eta /c-\alpha
\right) z}\frac{z}{r}\mathbb{P}\left( X_{r}\in dz\right)  \\
&=&\frac{1}{c-\eta /\alpha }\int_{0}^{\infty }\frac{z}{r}\mathbb{P}\left(
X_{r}\in dz\right) -\frac{e^{\left( \eta /c-\alpha \right) x}}{\left( c-\eta
/\alpha \right) }\int_{0}^{\infty }\frac{z}{r}\mathbb{P}\left( X_{r}\in
dz\right) +e^{\left( \eta /c-\alpha \right) x} \\
&=&\frac{1}{c-\eta /\alpha }\int_{0}^{\infty }\frac{z}{r}\mathbb{P}\left(
X_{r}\in dz\right) \left( 1-e^{\left( \eta /c-\alpha \right) x}\right)
+e^{\left( \eta /c-\alpha \right) x},
\end{eqnarray*}
where in the last equality, we used equation \eqref{int1}.\\
We refer the reader to \cite{kyprianou2014} for more details on spectrally negative L\'{e}vy processes and fluctuation identities.
\section{Main results}\label{mainresults}
\subsection{Distribution of longest negative duration}
In this subsection, we present our main results for the distributions of the longest negative duration $\bar{U}_\infty$, $\bar{U}_{\tau_b^+}$, $\bar{U}_{T_b^+}$ and $\bar{U}_{\me_q}$  where $\mathrm{e}_{q }$
denotes an exponential random variable with rate $q >0$ (independent of $X$) and $T_b^+$ is the first passage time under Poisson observations formally defined in \eqref{poissfirspassage}.\\
%\subsubsection{Longest duration up to an infinite time horizon}
We start our analysis with the distribution of $\bar{U}_{\infty}$,
that is, the longest negative excursion observed.
%up to an exponential time $\me_q$.
\begin{prop}\label{thm1}
For $r>0$, $x\in \mathbb{R}$ and $\e[X_1]>0$, we have
\begin{equation}\label{thm11}
\mathbb{P}_{x}\left( \bar{U}_{\infty }<r\right)=
\mathbb{E}\left[ X_{1}\right]\frac{\Lambda \left( x,r\right) }{\Upsilon
\left( r\right) },
\end{equation}
where
\begin{equation}\label{Upsilon}
\Upsilon \left( r\right) =\int_{0}^{\infty }\dfrac{z}{r}\mathbb{P}\left(
X_{r}\in \mathrm{d}z\right).
\end{equation}
\end{prop}
\begin{proof}
First, we assume that $X$ has bounded variation. As described in \cite{Yin2014ExactJL}, for $k \geq 2$, we define the time when $X$ reaches (resp., leaves) level $0$ from below (resp., above) for the $k^{\rm th}$ time by
$$L_k = \inf \left\lbrace t \geq R_{k-1}, \text{ } X_t < 0\right\rbrace ,$$
and
 $$ R_k =  \inf \left\lbrace t \geq L_k , \text{ } X_t = 0\right\rbrace ,$$
where $L_1 = \tau ^{-}_0$ and $R_1=\inf \left\lbrace t \geq L_1 , \textbf{ } X_t = 0\right\rbrace$. Let $D_{k}=R_{k}-L_{k}$ be the duration of the $k^{\rm th}$ negative excursion. Hence, we can then express $\bar{U}_{\infty }$ as
\begin{equation}\label{defi1}
\bar{U}_{\infty }=D_{(N_\infty)},
\end{equation}
where $D_{(n)}=\max \{\,D_{1},\ldots ,D_{n}\,\}$ is the $n$th order statistic, and
\begin{equation}\label{defN}
N_\infty  =\sup \left\{ k:\text{ }R_{k}<\infty \right\}.
\end{equation}
Note that $N_\infty$ has the following distribution: $\mathbb{P}_x\left( N_\infty  =0\right) =%
\mathbb{P}_x\left( \tau _{0}^{-}=\infty \right) \,$ and for $n=1,2,..,$
\begin{equation}\label{distN}
\mathbb{P}_x\left( N_\infty  =n\right) =\mathbb{P}\left( \tau _{0}^{-}=\infty \right)
\mathbb{P}_{x}\left( \tau _{0}^{-}<\infty \right) \left( \mathbb{P}\left(
\tau _{0}^{-}<\infty \right) \right) ^{n-1}.
\end{equation}
Indeed, the term $\mathbb{P}_{x}\left( \tau _{0}^{-}<\infty \right) $ appears above
as long we have at least one negative (below zero) excursion of the risk process.
Then the process returns to zero $n-1$ times and the pieces of these trajectories are i.i.d. since
$X_t$ is a L\'evy process. Finally, the last excursion from zero never returns to zero and hence
it identifies $N_\infty$.

The stationarity and independence of increments of $X$ imply that given
$\{N_\infty =n\}$, the random variables
$\{D_{k},k=1,\cdots,n\}$ are mutually independent and $\{D_{k},k=2,\ldots ,n\}$ are identically distributed.
We start from the key observation:
\begin{eqnarray}\label{eqkq}
\mathbb{P}_{x}\left( \bar{U}_{\infty }<r \right)  &=&\mathbb{P}%
_{x}\left( N_\infty  =0 \right) +\mathbb{P}_{x}\left( \bar{U}_{\infty }<r,\tau _{0}^{-}<\infty \right) \notag  \\
&=& \e\left[X_1 \right]W(x)+\mathbb{P}%
_{x}\left( \bar{U}_{\infty }<r,\tau _{0}^{-}<\infty \right).
\end{eqnarray}
%where the second equality follows from \eqref{E:classicalruinprobaX}.
Then,
\begin{align}\label{case1}
\mathbb{P}_{x}\left( \bar{U}_{\infty }<r,\tau _{0}^{-}<\infty \right)=&\mathbb{P}_{x}\left( D_{(N_\infty )}<r\right) \notag \\
&=\sum_{n=1}^{\infty }\mathbb{P}_{x}\left( N_\infty  =n\right) \mathbb{P}_{x}\left(
D_{(n)}<r|N_\infty =n\right)\notag  \\
&=\sum_{n=1}^{\infty }\mathbb{P}_{x}\left( N_\infty  =n\right) \mathbb{P}_{x}\left(
D_{1}<r,\cdots,D_{n}<r|N_\infty =n\right) \notag \\
&=\sum_{n=1}^{\infty }\mathbb{P}_{x}\left( N_\infty =n\right) \mathbb{P}_{x}\left(
D_{1}<r|R_{n}<\infty ,L_{n+1}=\infty \right) \\
&\times \mathbb{P}_x\left(
D_{2}<r|R_{n}<\infty ,L_{n+1}=\infty \right) ^{n-1},\notag
\end{align}
where, similar to \cite{Yin2014ExactJL}, we have
\begin{equation}\label{D1}
\mathbb{P}_{x}\left( D_{1}<r|R_{n}<\infty ,L_{n+1}=\infty \right) =\frac{%
\mathbb{E}_{x}\left[ \mathbb{P}_{X_{\tau _{0}^{-}}}\left(
\hat{\tau}_{0}^{+}<r\right) \mathbf{1}_{\left\{ \tau _{0}^{-}<\infty \right\} }\right]
}{\mathbb{P}_{x}\left( \tau _{0}^{-}<\infty \right) }
\end{equation}
and%
\begin{equation}\label{D2}
\mathbb{P}_x \left( D_{2}<r|R_{n}<\infty ,L_{n+1}=\infty \right) =\frac{%
\mathbb{E}\left[ \mathbb{P}_{X_{\tau _{0}^{-}}}\left( \hat{\tau}_{0}^{+}<r\right)
\mathbf{1}_{\left\{ \tau _{0}^{-}<\infty \right\} }\right] }{\mathbb{P}\left( \tau _{0}^{-}<\infty \right) }.
\end{equation}
Above, the subscript $X_{\tau_0^-}$
means that we are dealing here with an event for a copy of the process
$X$ now starting at the point $X_{\tau_0}$
at time $t=0$, that is independent of the segment $\{X_t, t\in [0,\tau_0^-)\}$ of the
original process. To underline this independent copy we add hat to the original counterparts.
In particular, above, $\hat{\tau}_{0}^{+}$ is a first passage time of zero of this independent copy.
Similarly, we will treat initial position $X_{\tau}$ for other stopping times $\tau$.
Using Eq. $(22)$ in Lkabous et al. \cite{lkabousetal2016} for $\delta =0$, we
have
\begin{equation*}
\mathbb{E}_{x}\left[ \mathbb{P}_{X_{\tau _{0}^{-}}}\left( \hat{\tau}
_{0}^{+}<r\right) \mathbf{1}_{\left\{ \tau _{0}^{-}<\infty \right\} }\right]
=\Lambda \left( x,r\right) -W\left( x\right) \Upsilon \left( r\right)
\end{equation*}
and
\begin{equation*}
\mathbb{E}\left[ \mathbb{P}_{X_{\tau _{0}^{-}}}\left( \hat{\tau} _{0}^{+}<r\right)
\mathbf{1}_{\left\{ \tau _{0}^{-}<\infty \right\} }\right] =1-W\left(
0\right) \Upsilon \left( r\right) ,
\end{equation*}
where we used the fact that $\Lambda \left( 0,r\right) =1$ and $\Upsilon$ is defined in \eqref{Upsilon}.
Plugging the above two expressions in \eqref{case1}, we obtain
\begin{align}\label{recover}
\mathbb{P}_{x}\left( \bar{U}_{\infty }<r,\tau _{0}^{-}<\infty \right)  =&\mathbb{P}\left( \tau _{0}^{-}=\infty \right) \mathbb{E}_{x}\left[
\mathbb{P}_{X_{\tau _{0}^{-}}}\left( \hat{\tau} _{0}^{+}<r\right) \right]
\sum_{n=1}^{\infty }\mathbb{E}\left[ \mathbb{P}_{X_{\tau _{0}^{-}}}\left(
\hat{\tau} _{0}^{+}<r\right) \mathbf{1}_{\left\{ \tau _{0}^{-}<\infty \right\} }\right] ^{n-1} \notag \\
&=\mathbb{P}\left( \tau _{0}^{-}=\infty \right) \mathbb{E}_{x}\left[
\mathbb{P}_{X_{\tau _{0}^{-}}}\left( \hat{\tau} _{0}^{+}<r\right) \right] \frac{1}{%
1-\mathbb{E}\left[ \mathbb{P}_{X_{\tau _{0}^{-}}}\left( \hat{\tau}
_{0}^{+}<r\right) \mathbf{1}_{\left\{ \tau _{0}^{-}<\infty \right\} }\right]
}\notag \\
&=W\left( 0\right) \mathbb{E}\left[ X_{1}\right] \left( \Lambda \left(
x,r\right) -W\left( x\right) \Upsilon \left( r\right) \right) \frac{1}{W\left(
0\right) \Upsilon \left( r\right) }\notag  \\
&=\mathbb{E}\left[ X_{1}\right] \left( \frac{\Lambda \left( x,r\right) }{%
\Upsilon \left( r\right) }-W\left( x\right) \right),
\end{align}
where in the last equality we used the fact that $W(0) > 0$. The resutls follows by substiting the last expression in \eqref{case1}.

	Now, if $X$ has paths of unbounded variation, we consider an approximation approach similar to \cite{Yin2014ExactJL}. More specifically, for $\epsilon >0$, let
	\begin{align*}
	L_1(\epsilon) = \tau ^{-}_{-\epsilon} = \inf \left\{t\ge 0: X_t <-\epsilon \right\},
	\end{align*}
	and
	\begin{align*}
	R_1(\epsilon)=\inf\left\{t \geq L_1 (\epsilon) :  X_t = 0\right\}.
	\end{align*}
	Recursively, we define two sequences of stopping times $\left\{L_k(\epsilon)\right\}_{k \geq 1}$ and $\left\{R_k(\epsilon)\right\}_{k \geq 1}$ as follows. For $k \geq 2$, let
	$$L_k  (\epsilon) = \inf \left\lbrace t \geq R_{k-1}(\epsilon):  X_t < -\epsilon \right\rbrace $$
	and
	$$ R_k (\epsilon)  =  \inf \left\lbrace t \geq L_k (\epsilon)  :  X_t = 0\right\rbrace.$$
 Let $D^{(\epsilon)}_{k}=R_{k}(\epsilon)-L_{k} (\epsilon)$ represents the duration of the $k^{\rm th}$ negative excursion below the level $-\epsilon $.
Let
	\begin{equation*}
	N^{(\epsilon)}_\infty =\sup \left\{ k: L_{k}(\epsilon) <0 \right\},
	\end{equation*}
and we write
\begin{equation}\label{defi1}
\bar{U}^{(\epsilon)}_{\infty }=D_{(N^{(\epsilon)}_\infty )}.
\end{equation}
Then
\begin{align*}
\mathbb{P}_{x}\left( \bar{U}^{(\epsilon)}_{\infty }<r,\tau _{-\epsilon }^{-}<\infty \right)  =&\mathbb{P}\left( \tau _{-\epsilon }^{-}=\infty \right) \mathbb{E}_{x}\left[ \mathbb{P}_{X_{\tau _{-\epsilon }^{-}}}\left( \hat{\tau} _{0}^{+}<r\right) \right]
\sum_{n=1}^{\infty }\mathbb{E}\left[ \mathbb{P}_{X_{\tau _{-\epsilon }^{-}}}\left(
\hat{\tau} _{0}^{+}<r\right) \mathbf{1}_{\left\{ \tau _{-\epsilon }^{-}<\infty \right\} }\right] ^{n-1} \\
&=\mathbb{P}\left( \tau _{-\epsilon }^{-}=\infty \right) \mathbb{E}_{x}\left[
\mathbb{P}_{X_{\tau _{-\epsilon }^{-}}}\left( \hat{\tau} _{0}^{+}<r\right) \right] \frac{1}{%
1-\mathbb{E}\left[ \mathbb{P}_{X_{\tau _{-\epsilon }^{-}}}\left( \hat{\tau}
_{0}^{+}<r\right) \mathbf{1}_{\left\{ \tau _{-\epsilon }^{-}<\infty \right\} }\right]
} \\
&=W\left( \epsilon \right) \mathbb{E}\left[ X_{1}\right] \left( \Lambda^{(\epsilon)} \left(
x-\epsilon ,r\right) -W\left(x \right)\Upsilon^{(\epsilon)} \left( r\right) \right) \frac{1}{ 1-\Lambda^{(\epsilon)} \left(0,z\right) +\Upsilon^{(\epsilon)}  \left( r\right) }, \\
\end{align*}
where
$$ \Lambda^{(\epsilon)} \left(x,z\right) =\int_{\epsilon}^{\infty }W\left(x+u\right) \frac{u}{z}\mathbb{P}\left( X_{z}\in \mathrm{d}u\right), $$
and
$$ \Upsilon ^{(\epsilon)}\left( r\right)=\int_{\epsilon}^{\infty } \frac{u}{z}\mathbb{P}\left( X_{z}\in \mathrm{d}u\right).$$
Thus
\begin{align*}
\mathbb{P}_{x}\left( \bar{U}_{\infty }<r,\tau _{0 }^{-}<\infty \right) &=\lim_{\epsilon \rightarrow 0 } \mathbb{P}_{x}\left( \bar{U}^{(\epsilon)}_{\infty }<r,\tau _{-\epsilon }^{-}<\infty \right)\\ &=
\lim_{\epsilon \rightarrow 0 } \ \frac{W\left( \epsilon \right) \mathbb{E}\left[ X_{1}\right] \left( \Lambda^{(\epsilon)} \left(
x-\epsilon ,r\right) -W\left(x \right)\Upsilon^{(\epsilon)} \left( r\right) \right)}{ 1-\Lambda^{(\epsilon)} \left(0,z\right) +  W\left( \epsilon\right)\Upsilon^{(\epsilon)}  \left( r\right) } \\
&=\mathbb{E}\left[ X_{1}\right] \left( \frac{\Lambda \left( x,r\right) }{%
\Upsilon \left( r\right) }-W\left( x\right) \right),
\end{align*}
where in the last equality, we used $ \lim_{\epsilon \rightarrow 0 }  \frac{1- \Lambda^{(\epsilon)} \left(0,r\right)}{W(\epsilon)} =0$ as showed in Eq. $(14)$ of \cite{loeffenetal2013}.
\end{proof}
The distribution of $\bar{U}_{\infty}$ in Proposition \ref{thm1} coincides with the expression of the probability of Parisian ruin with delay $r>0$ obtained in \cite{loeffenetal2013} with
\begin{eqnarray}\label{SNLPPr2}
\mathbb{P}_{x}\left( \bar{U}_{\infty }>r\right)=\mathbb{P}_{x}\left( \kappa _{r}<\infty \right),\end{eqnarray}
where $\kappa_r$ is the Parisian ruin time and also the first passage time of $U$ above $r$ (see Figure~\ref{fig2}).
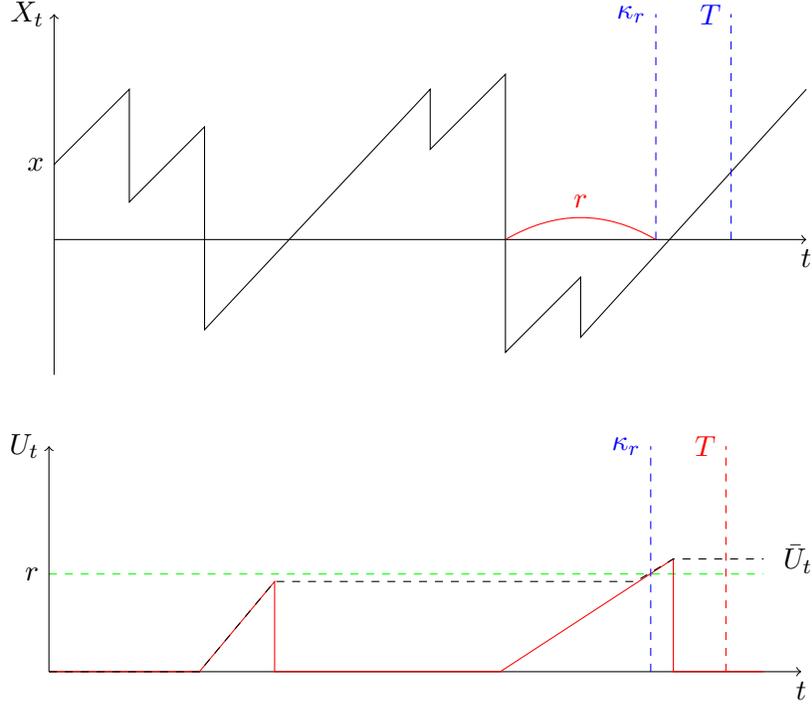
\begin{figure}[h!]
\begin{center}
\begin{tikzpicture}[domain=0:5]
  \draw[ thin,color=gray] (-0.1,-1.1)  (3.9,3.9);
\draw[->] (0,0) -- (10,0) node[below] {$t$};
 \draw[-] (0,1) -- (0,1) node[left] {$x$};
 %\draw[-][color=red] (0,1.8) -- (0,1.8) node[left] {$a$};
%\draw[-][color=red] (0,1.8) -- (10,1.8) ;
\draw[->] (0,-1.8) -- (0,3) node[left] {$X_t$};
\draw[color=black] (0,1) -- (1,2) --(1,0.5)--(2,1.5)--(2,-1.2)--(5,2)--(5,1.2)--(6,2.2)--(6,-1.5)--(7,-0.5)--(7,-1.3)--(10,2);
\path[-]  [red](6,0) edge [bend left] node[above] {$r$} (8,0);
  \draw[-][blue,dashed](8,0) -- (8,3) node[left] {$\kappa_r$};
      \draw[-][blue,dashed](9,0) -- (9,3) node[left] {$T$};
\end{tikzpicture}
 % \path[-]  [red](6.5,0) edge [bend left] node[above] {$r$} (8.5,0);
  % \draw[-][blue](0,0) -- (0,0) node[left] {$\tilde{\kappa}^{\lambda}_r$};
  \begin{tikzpicture}[domain=0:5]
  \draw[ thin,color=gray] (-0.1,-1.1)  (3.9,3.9);
\draw[->] (0,0) -- (10,0) node[below] {$t$};
 \draw[-] (0,1.3) -- (0,1.3) node[left] {$r$};
\draw[-][color=green,dashed] (0,1.3) -- (9.5,1.3) node[left] {};
 %\draw[-][color=red] (0,1.8) -- (0,1.8) node[left] {$a$};
%\draw[-][color=red] (0,1.8) -- (10,1.8) ;
\draw[->] (0,0) -- (0,3) node[left] {$U_t$};
\draw[color=red] (0,0) -- (0,0) --(0,0)--(2,0)--(3,1.2)--(3,1.2)--(3,0)--(5,0)--(6,0) --(8.3,1.5)--(8.3,0)--(8.3,0)--(9.5,0);
\draw[color=black,dashed] (0,0) -- (0,0) --(0,0)--(2,0)--(3,1.2)--(7.8,1.2)--(8.3,1.5)--(9.5,1.5)node[right] {$\text{            }\bar{U}_t$}; ;
  \draw[-][blue,dashed](8,0) -- (8,3) node[left] {$\kappa_r$};
    \draw[-][red,dashed](9,0) -- (9,3) node[left] {$T$};
\end{tikzpicture}
\caption{Sample path of $U$ (red line) and corresponding running longest $\bar{U}$ (dashed black line).}
\label{fig2}
\end{center}
\end{figure}
%where $U_t =t-g_t $ with $g_t =\sup \left\{ 0\leq s\leq t : X_s \geq 0 \right\}$.
%We will show in this section that using the main results of this paper one can recover known results concerning this ruin time.
%We recall that, for a spectrally negative L\'evy insurance risk process $X$, Loeffen et al.\ \cite{loeffenetal2013} obtained an explicit expression for the probability of Parisian ruin, that is, for $\e\left[X_1\right] >0$ and $x \in \reals$,
%\begin{equation}\label{SNLPPr}
%\p_{x}\left(\kappa_{r}<\infty \right) = 1-\e[X_{1}]\frac{\Lambda(x,r)}{\Upsilon
%\left( r\right)} .
%\end{equation}
%From Theorem \ref{thm1}, we recover the above expression for the probability of Parisian ruin as the survival distribution function of $\bar{U}_{\infty }$ since
%\begin{equation}\label{SNLPPr}
%\p_{x}\left(\bar{U}_\infty >r \right)=\p_{x}\left(\kappa_{r}<\infty \right),
%\end{equation}
%\textbf{ We would like to find the expression of $\p_{x}\left(\kappa_{r}<\infty \right)$ using $\p_{x}\left(\bar{U}_\infty \in \md y \right)$}.
%Using \eqref{Uinf}, it should be clear that Then
%and such that
%\begin{equation*}
%\p_{x}\left(\bar{U}_\infty <y,\tau_{0}^{-}<\infty \right)=\mathbb{P}\left( \kappa _{r}=\infty \right) \mathbb{E}_{x}\left[ \mathbb{P}%
%_{X_{\tau _{0}^{-}}}\left( \tau _{0}^{+}<r\right) \ind_{\left\{ \tau
%_{0}^{-}<\infty \right\} }\right] =\mathbb{E}\left[ X_{1}\right] \frac{%
%\mathbb{E}_{x}\left[ \Lambda \left( X_{\tau _{0}^{-}},r\right) \ind_{\left\{
%\tau _{0}^{-}<\infty \right\} }\right] }{\int_{0}^{\infty }\frac{z}{r}%
%\mathbb{P}\left( X_{r}\in \mathrm{d}z\right) }
%\end{equation*}
Similarly, one can link many results available in the literature with the distribution of the largest negative excursion. Indeed, for $x\leq b$, we have
\begin{eqnarray*}
\mathbb{P}_{x}\left( \bar{U}_{\tau^{+}_b  }<r\right) =\p_{x}\left(\tau^{+}_b <\kappa_{r} \right)= \frac{\Lambda(x,r)}{\Lambda(b,r)},
\end{eqnarray*}
where the second equality follows from Lkabous et al. \cite{lkabousetal2016} for $\delta=0$.
%\begin{equation}\label{SNLPPrb}
%\p_{x}\left(\tau^{+}_b <\kappa_{r} \right) = \frac{\Lambda(x,r)}{\Lambda(b,r)},
%\end{equation}
%\mathbb{P}_{x}\left(
%\tau^{+}_b <\tau _{0}^{-} \right) +\mathbb{P}_{x}\left( \bar{U}_{\tau^{+}_b  }<r,\tau
%_{0}^{-}<\infty \right)   \label{SNLPPr} \\
%&=&\dfrac{ W\left( x\right)}{ W\left( b\right)} + \frac{\Lambda \left( x,r\right) }{\Lambda \left( x,r\right)}-\dfrac{ W\left( x\right)}{ W\left( b\right)}  \\
We can also recover the Laplace transform of $\kappa_{r}$ derived in \cite{lkabousetal2016} and \cite{loeffenetal2017}, that is,
\begin{eqnarray}
\mathbb{P}_{x}\left( \bar{U}_{\mathrm{e}_q}>r\right)
&=&\mathbb{P}_{x}\left( \kappa _{r}<\mathrm{e}_{q}\right)   \nonumber \\
&=&\mathrm{e}^{-qr}Z_{q}\left( x\right) +q\mathrm{e}^{-qr}\int_{0}^{r}%
\Lambda ^{\left( q\right) }\left( x,s\right) \mathrm{d}s\nonumber  \\
&&-\frac{\mathrm{e}^{-qr}\Lambda ^{\left( q\right) }\left( x,r\right) }{%
\int_{0}^{\infty }\mathrm{e}^{\Phi _{q}z}(z/r)\mathbb{P}\left( X_{r}\in
\mathrm{d}z\right) }\left( \frac{q}{\Phi _{q}}+q\int_{0}^{r}\int_{0}^{\infty
}\mathrm{e}^{\Phi _{q}z}\frac{z}{r}\mathbb{P}\left( X_{r}\in \mathrm{d}%
z\right) \mathrm{d}s\right).
\end{eqnarray}
%\subsection{Parisian ruin with exponential delay}
%In the case where the delay $r$ is an exponential random variable $\me_q$, and that an independent copy of $\me_q$ is assigned to each negative excursion, we obtain the time of Parisian ruin with exponential delays $\kappa ^{q}$ defined as
%\begin{equation}
%\kappa ^{q}=\inf \left\{ t>0:U_{t}>\mathrm{e}_{q}^{g_{t}}\right\} ,
%\label{kappaq}
%\end{equation}
%where$\mathrm{e} _{q}^{g_{t}}$ is an exponentially distributed random variable (independent of $X$) with rate $q>0$.
%. An expression for the probability of Parisian ruin with exponential delays was first given in \cite{landriaultetal2011}, that is, for $\mathbb{E}
%[X_{1}]>0$, $q>0$ and $x\in \mathbb{R}$,
%\begin{equation}
%\mathbb{P}_{x}\left( \kappa ^{q}<\infty \right) =1-\mathbb{E}\left[ X_{1}\right] \frac{\Phi _{q}}{q}Z\left( x,\Phi _{q}\right) .  \label{Pruine1}
%\end{equation}
%We can readily recover (\ref{Pruine1}) using our result.
More generally, for a finite-horizon time $t$, we have
\begin{equation*}
\p_{x}\left(\bar{U}_t >r \right)=\p_{x}\left(\kappa_{r}\leq t \right).
\end{equation*}
Finally, the Laplace transform of $\bar{U}_{\mathrm{e}_{\theta }}$ is given by
\begin{equation*}
\mathbb{E}_{x}\left[ \mathrm{e}^{-q\bar{U}_{\mathrm{e}_{\theta }}}\right]
=\int_{0}^{\infty }q\mathrm{e}^{-qr}\mathbb{P}_{x}\left( \bar{U}_{\mathrm{e}%
_{\theta }}<r\right) \mathrm{d}r=\int_{0}^{\infty }q\mathrm{e}^{-qr}\mathbb{P}%
_{x}\left( \kappa _{r}>\mathrm{e}_{\theta }\right) \mathrm{d}r.
\end{equation*}
%Starting from $x=0$ and given that ruin occurs, Dickson \cite{dickson1992distribution} showed that $\tau^-_0$ and $D_1$ have the same distribution for the classical compound Poisson risk process. Thus, for the process in \ref{CL} and using in Eq. (D.17) in \cite{drekic2009joint}, we have
%\begin{equation*}
%\p(D_1\in \md z)=\p(\tau^-_0 \in \md z)=\eta\me^{-\eta r} \me^{ -(\eta+c\alpha)z}\left( I_0 (2\sqrt{r\eta \alpha}z)-I_2 (2\sqrt{r\eta \alpha}z)  \right) \mathrm{d}z .
%\end{equation*}
%\end{remark}
%%%%%%%%%%%%%%%%%%%%%%%%%%%%%%%%
%%%%%%%%%%%%%%%%%%%%%%%%%%%%%%%%%%%%%%%
%\subsubsection{Longest duration up to the first Poissonian time}
 Our approach can be extended to first Poissonian passage times. We define the first time the surplus $X$ is observed above $b$ at a Poisson arrival time by
\begin{equation}\label{poissfirspassage}
T^+_b := \min \{T_i : X_{T_i} > b, i \in \mathbb{N}\}
\end{equation}
with the convention $\sup \emptyset = 0$, where $\{T_i\}$ are points of the Poisson process with intensity $\alpha>0$.
%\begin{theorem}\label{thm2.2}
We start by identifying
$$\mathbb{P}_{x}\left( \bar{U}_{T _{b}^{+}}<r\right).$$
Letting $\theta=0$ in Eq. (17) in \cite{albrecheretal2016}, we have
\begin{equation}\label{f}
\mathbb{P}_{x}\left( T  _{b}^{+}<\tau _{0}^{-}\right)=\dfrac{\alpha W(x)}{\Phi_\alpha Z(b,\Phi_\alpha)}.
\end{equation}
Moreover, note that
\begin{eqnarray}
\mathbb{P}_{x}\left( \bar{U}_{T _{b}^{+}}\in dy\right) &=&
\mathbb{P}_{x}\left( T_{b}^{+}<\tau _{0}^{-}\right) \delta _{0}\left( dy\right) +%
\mathbb{P}_{x}\left( \bar{U}_{T_{b}^{+}}\in dy,\tau _{0}^{-}<T _{b}^{+}
\right)  \\
&=& \dfrac{\alpha W(x)}{\Phi_\alpha Z(b,\Phi_\alpha)} \delta _{0}\left( dy\right) +%
\mathbb{P}_{x}\left( \bar{U}_{T _{b}^{+}}\in dy,\tau _{0}^{-}<T _{b}^{+}
\right)
\end{eqnarray}
and hence
\begin{equation}\label{nowy2}
\mathbb{P}_{x}\left( \bar{U}_{T_{b}^{+}}< r\right) = \dfrac{\alpha W(x)}{\Phi_\alpha Z(b,\Phi_\alpha)}  + \mathbb{P}_{x}\left( \bar{U}_{T_{b}^{+}}<r,\tau_{0}^{-}<T_{b}^{+}
\right).
\end{equation}
%We will prove now that
%\begin{equation}\label{Uuptob}
%\mathbb{P}_{x}\left( \bar{U}_{T _{b}^{+}}<r,\tau _{0}^{-}<T _{b}^{+} \right)=....
%\end{equation}
Similarly to the previous proof, the random variable $\bar{U}_{T _{b}^{+} }$ can
expression as
\begin{equation}\label{defi2}
\bar{U}_{T _{b}^{+} }=D_{(N_{T _{b}^{+}})},
\end{equation}
where
\begin{equation*}
N_{T _{b}^{+}}=\sup \left\{ k:\text{ }R_{k}<T _{b}^{+}\right\}
\end{equation*}%
such that $\mathbb{P}_x\left( N_{T _{b}^{+}}=0\right) =%
\mathbb{P}_x\left( T _{b}^{+}<\tau _{0}^{-}\right) \,$\ and for $n=1,2,..,$
\begin{equation*}
\mathbb{P}_x\left( N_{T _{b}^{+}}=n\right) =\mathbb{P}_{x}\left( T _{b}^{+}<\tau
_{0}^{-}\right) \mathbb{P}_{x}\left( \tau _{0}^{-}<T _{b}^{+}\right)
\left( \mathbb{P}\left( \tau _{0}^{-}<T _{b}^{+}\right) \right) ^{n-1}%
\text{ .}
\end{equation*}
%The stationary and independence of increments of $X$ imply that given $\{N_{T _{b}^{+}}=n\}$, the sequence $\{D_{k},k=1,\ldots,n\}$ are mutually independent and $\{D_{k},k=2,\ldots,n\}$ are identically distributed.
Hence
\begin{align}
\mathbb{P}_{x}\left( \bar{U}_{T _{b}^{+}}<r,\tau _{0}^{-}<T
_{b}^{+}\right)
=&\mathbb{P}_{x}\left( \max_{k=1,..,N_{T _{b}^{+}}}D_{k}<r,\tau _{0}^{-}<\tau
_{b}^{+}\right)  \nonumber\\
&=\sum_{n=1}^{\infty }\mathbb{P}_{x}\left( N_{T _{b}^{+}}=n\right) \mathbb{P}_{x}\left(
\max_{k=1,..,n}D_{k}<r|N_{T _{b}^{+}}=n\right)\nonumber  \\
&=\sum_{n=1}^{\infty }\mathbb{P}_{x}\left( N_{T _{b}^{+}}=n\right) \mathbb{P}_{x}\left(
D_{1}<r,...,D_{n}<r|N_{T _{b}^{+}}=n\right)  \nonumber\\
&=\sum_{n=1}^{\infty }\mathbb{P}_{x}\left( N_{T _{b}^{+}}=n\right) \mathbb{P}_{x}\left(
D_{1}<r,|R_{n}<T _{b}^{+},L_{n+1}>T _{b}^{+} \right)\notag \\ &\times \mathbb{P}\left(
D_{2}<r|R_{n}<T _{b}^{+},L_{n+1}>T _{b}^{+} \right) ^{n-1},\label{nowy3}
\end{align}%
where
\begin{equation}\label{nowy4}
\mathbb{P}_{x}\left( D_{1}<r|R_{n}<T _{b}^{+},L_{n+1}>T _{b}^{+} \right) =%
\frac{\mathbb{E}_{x}\left[ \mathbb{P}_{X_{\tau _{0}^{-}}}\left( \hat{\tau}
_{0}^{+}<r\right) \ind_{\left\{ \tau _{0}^{-}<T _{b}^{+}\right\} }\right] }{%
\mathbb{P}_{x}\left( \tau _{0}^{-}<T _{b}^{+}\right) }
\end{equation}%
and
\begin{equation}\label{nowy5}
\mathbb{P}\left(D_{2}<r|R_{n}<T _{b}^{+},L_{n+1}>T _{b}^{+}\right) =
\frac{\mathbb{E}\left[ \mathbb{P}_{X_{\tau _{0}^{-}}}\left( \hat{\tau}_{0}^{+}<r\right)
\ind_{\left\{ \tau _{0}^{-}<T _{b}^{+}\right\}}
\right]}{ \mathbb{P}
\left(\tau _{0}^{-}<T _{b}^{+}\right)}.
\end{equation}
Further, by \eqref{nowy}
\[\mathbb{E}_x\left[ \mathbb{P}_{X_{\tau _{0}^{-}}}\left( \hat{\tau}_{0}^{+}<r\right)
\ind_{\left\{ \tau _{0}^{-}<T _{b}^{+}\right\} } \right]=\mathbb{E}_{x}\left[ \Lambda \left( X_{\tau _{0}^{-}},r\right) \ind_{\left\{
\tau _{0}^{-}<T _{b}^{+}\right\} }\right].\]
%but first, we need to find $\mathbb{E}_x \left[  W \left( X_{\tau _{0}^{-}}+z\right) \mathbf{1}_{\{\tau^{-}_0 <T^{+}_b \}} \right]$. From \cite{albrecheretal2016}, we have
%\begin{equation}\label{ID}
%\mathbb{E}_{x}\left[ e^{\theta X_{\tau _{0}^{-}}}\mathbf{1}_{\{\tau
%_{0}^{-}<T_{b}^{+}\}}\right] =Z\left( x,\theta \right) -W\left( x\right)
%\left( \frac{\psi \left( \theta \right) }{\theta -\Phi _{\lambda }}-\frac{%
%\lambda }{\theta -\Phi _{\lambda }}\frac{Z\left( b,\theta \right) }{Z\left(
%b,\Phi _{\lambda }\right) }\right)
%\end{equation}
%We need to inverse with respect to $\theta$ to obtain $\mathbb{P}_{x}\left(X_{\tau _{0}^{-}} \in \md y ,\tau_{0}^{-}< T_{b}^{+} \right) $.
We will find now $\mathbb{P}_{x}\left(X_{\tau _{0}^{-}} \in \md y ,\tau_{0}^{-}< T_{b}^{+} \right)$
which produces $\mathbb{E}_{x}\left[ \Lambda \left( X_{\tau _{0}^{-}},r\right) \ind_{\left\{
\tau _{0}^{-}<T _{b}^{+}\right\} }\right]$
and $\mathbb{P}_x\left(\tau _{0}^{-}<T _{b}^{+}\right)$.
Note that
\begin{equation}\label{nowy6}\mathbb{P}_{x}\left(X_{\tau _{0}^{-}} \in \md y ,\tau_{0}^{-}< T_{b}^{+} \right) =\mathbb{P}_{x}\left(X_{\tau _{0}^{-}} \in \md y\right)
-\int_b^\infty\mathbb{P}_{x}\left(X_{T_{b}^{+}} \in \md z ,T_{b}^{+}<\tau_{0}^{-} \right)
\mathbb{P}_{z}\left(X_{\tau _{0}^{-}} \in \md y\right),\end{equation}
%Now
%\begin{equation}\label{nowy7}
%\mathbb{P}_{z}\left(X_{\tau _{0}^{-}} \in \md y\right) =\lambda \int_0^\infty (\mathrm{e}^{\Phi y}W(z)-W(z+y))F(z-\md y) \md z
%\end{equation}
%by \cite[Thm. 5.5., p. 41]{kyprianou2013}
where
\begin{equation}\label{nowy8}
\mathbb{P}_{x}\left(X_{T_{b}^{+}} \in \md z ,T_{b}^{+}<\tau_{0}^{-} \right)=
\alpha \frac{W^{(0,\alpha)}(b-z)}{W^{(0,\alpha)}(b)}W(x)-W(x-z)
\end{equation}
by \cite[Cor. 3.1]{landriault2018distribution},
for
\[W^{(0,\alpha)}(x)=\dfrac{(\Phi_\alpha -\Phi_0)}{\alpha}Z(x,\Phi_\alpha).\]
%\begin{theorem}\label{thm2b}
%For $\alpha>0$, $b,r>0$ and $x\leq b$,

In this way, we obtain $\mathbb{P}_{x}\left( \bar{U}_{T _{b}^{+}}<r\right)$
given in \eqref{nowy2}, where $\mathbb{P}_{x}\left( \bar{U}_{T _{b}^{+}}<r,\tau _{0}^{-}<T_{b}^{+}\right)$ is identified by identities \eqref{nowy3}-\eqref{nowy8}.
%\end{theorem}
%\subsection{Distribution of the shortest negative duration }
\subsection{Shortest duration up to an infinite time horizon}
In this subsection, we will derive the distribution of the shortest negative duration that we denote by $\underline{U}_{\infty }$ which is only for relevant bounded variation processes. First, we notice that $\underline{U}_{\infty }$ can be represented as
\begin{equation}\label{defi1}
\underline{U}_{\infty }=D_{(1)},
\end{equation}
where $D_{(1)}=\min \{ D_{1},\ldots ,D_{N_\infty} \}$.
%gain, we start from the observation:
%\begin{eqnarray}\label{eqkq}
%\mathbb{P}_{x}\left( \underline{U}_{\infty }<r\right)  &=&\mathbb{P}%
%_{x}\left( N_\infty =0 \right) +\mathbb{P}
%_{x}\left( \underline{U}_{\infty }<r,\tau _{0}^{-}<\infty \right) \notag  \\
%&=& \e\left[X_1 \right]W(x)+\mathbb{P}%
%_{x}\left( \underline{U}_{\infty }<r,\tau _{0}^{-}<\infty \right).
%\end{eqnarray}
Then, we have the following expression for $\mathbb{P}_{x}\left( \underline{U}_{\infty }>r ,\tau _{0}^{-}<\infty \right)$.
%where $\delta_0(\cdot)$ is the Dirac mass at $0$ and the second equality follows from \eqref{E:classicalruinprobaX}.
\begin{theorem}\label{thmmin}
Let $r>0$, $x\in \mathbb{R}$ and $\e[X_1]>0$. If $W(0)>0$ (or equivalently $X$ has paths
of bounded variation), then
\begin{align*}
\mathbb{P}_{x}\left( \underline{U}_{\infty }>r,\tau _{0}^{-}<\infty  \right)  =&
W\left( 0\right) \mathbb{E}\left[ X_{1}\right]\frac{1-W(x)(\e[X_1]-\Upsilon \left( r\right)) -\Lambda \left( x,r\right) }{1-W(0)( \Upsilon \left( r\right)-\e[X_1]) }.
\end{align*}
\end{theorem}
\begin{proof}
We have
\begin{align}\label{eq1}
\mathbb{P}_{x}\left( \underline{U}_{\infty }>r ,\tau _{0}^{-}<\infty \right)=&\mathbb{P}_{x}\left( D_{(1)}>r \right)\notag
 \\
&=\sum_{n=1}^{\infty }\mathbb{P}_{x}\left( N_\infty =n\right) \mathbb{P}_{x}\left(
D_{(1)}>r|N_\infty =n\right)\notag  \\
&=\sum_{n=1}^{\infty }\mathbb{P}_{x}\left( N_\infty =n\right) \mathbb{P}_{x}\left(
D_{1}>r,\cdots,D_{n}>r|N_\infty =n\right) \notag \\
&=\sum_{n=1}^{\infty }\mathbb{P}_{x}\left( N_\infty =n\right) \mathbb{P}_{x}\left(
D_{1}>r|R_{n}<\infty ,L_{n+1}=\infty \right) \\
&\times \mathbb{P}_x\left(
D_{2}>r|R_{n}<\infty ,L_{n+1}=\infty \right) ^{n-1},\notag
\end{align}
where
\begin{equation*}
\mathbb{P}_{x}\left( D_{1}>r|R_{n}<\infty ,L_{n+1}=\infty \right) =\frac{%
\mathbb{E}_{x}\left[ \mathbb{P}_{X_{\tau _{0}^{-}}}\left( \hat{\tau}
_{0}^{+}>r\right) \mathbf{1}_{\left\{ \tau _{0}^{-}<\infty \right\} }\right]
}{\mathbb{P}_{x}\left( \tau _{0}^{-}<\infty \right) }
\end{equation*}
and
\begin{equation*}
\mathbb{P}_x \left( D_{2}>r|R_{n}<\infty ,L_{n+1}=\infty \right) =\frac{%
\mathbb{E}\left[ \mathbb{P}_{X_{\tau _{0}^{-}}}\left( \hat{\tau} _{0}^{+}>r\right)
\mathbf{1}_{\left\{ \tau _{0}^{-}<\infty \right\} }\right] }{\mathbb{P}\left( \tau _{0}^{-}<\infty \right) }.
\end{equation*}
Using the fact that
\begin{equation*}
\mathbb{E}_{x}\left[ \mathbb{P}_{X_{\tau _{0}^{-}}}\left( \hat{\tau}
_{0}^{+}>r\right) \mathbf{1}_{\left\{ \tau _{0}^{-}<\infty \right\} }\right]
=1-\e[X_1] W\left( x\right)-\Lambda \left( x,r\right) +W\left( x\right) \Upsilon \left( r\right),
\end{equation*}
\begin{equation*}
\mathbb{E}\left[ \mathbb{P}_{X_{\tau _{0}^{-}}}\left( \hat{\tau} _{0}^{+}>r\right)
\mathbf{1}_{\left\{ \tau _{0}^{-}<\infty \right\} }\right]=W\left(
0\right) (\Upsilon \left( r\right) -\e[X_1]),
\end{equation*}
and plugging the above two expressions in \eqref{eq1}, we finally obtain
\begin{align*}
\mathbb{P}_{x}\left( \underline{U}_{\infty }>r ,\tau _{0}^{-}<\infty \right)  =&\mathbb{P}\left( \tau _{0}^{-}=\infty \right) \mathbb{E}_{x}\left[
\mathbb{P}_{X_{\tau _{0}^{-}}}\left( \hat{\tau} _{0}^{+}>r\right)\mathbf{1}_{\left\{ \tau _{0}^{-}<\infty \right\} } \right]
\sum_{n=1}^{\infty }\mathbb{E}\left[ \mathbb{P}_{X_{\tau _{0}^{-}}}\left(
\hat{\tau} _{0}^{+}>r\right) \mathbf{1}_{\left\{ \tau _{0}^{-}<\infty \right\} } \right] ^{n-1} \\
&=\mathbb{P}\left( \tau _{0}^{-}=\infty \right) \frac{ \mathbb{E}_{x}\left[
\mathbb{P}_{X_{\tau _{0}^{-}}}\left( \hat{\tau} _{0}^{+}>r\right)\mathbf{1}_{\left\{ \tau _{0}^{-}<\infty \right\} }  \right]}{1-\mathbb{E}\left[ \mathbb{P}_{X_{\tau _{0}^{-}}}\left( \hat{\tau}
_{0}^{+}>r\right) \mathbf{1}_{\left\{ \tau _{0}^{-}<\infty \right\} }\right]
} \\
&=W\left( 0\right) \mathbb{E}\left[ X_{1}\right]\frac{1-W(x)(\e[X_1]-\Upsilon \left( r\right)) -\Lambda \left( x,r\right) }{1-W(0)( \Upsilon \left( r\right)-\e[X_1]) }.
\end{align*}
\end{proof}

\begin{figure}[h]
\includegraphics[scale=1,height=8cm]{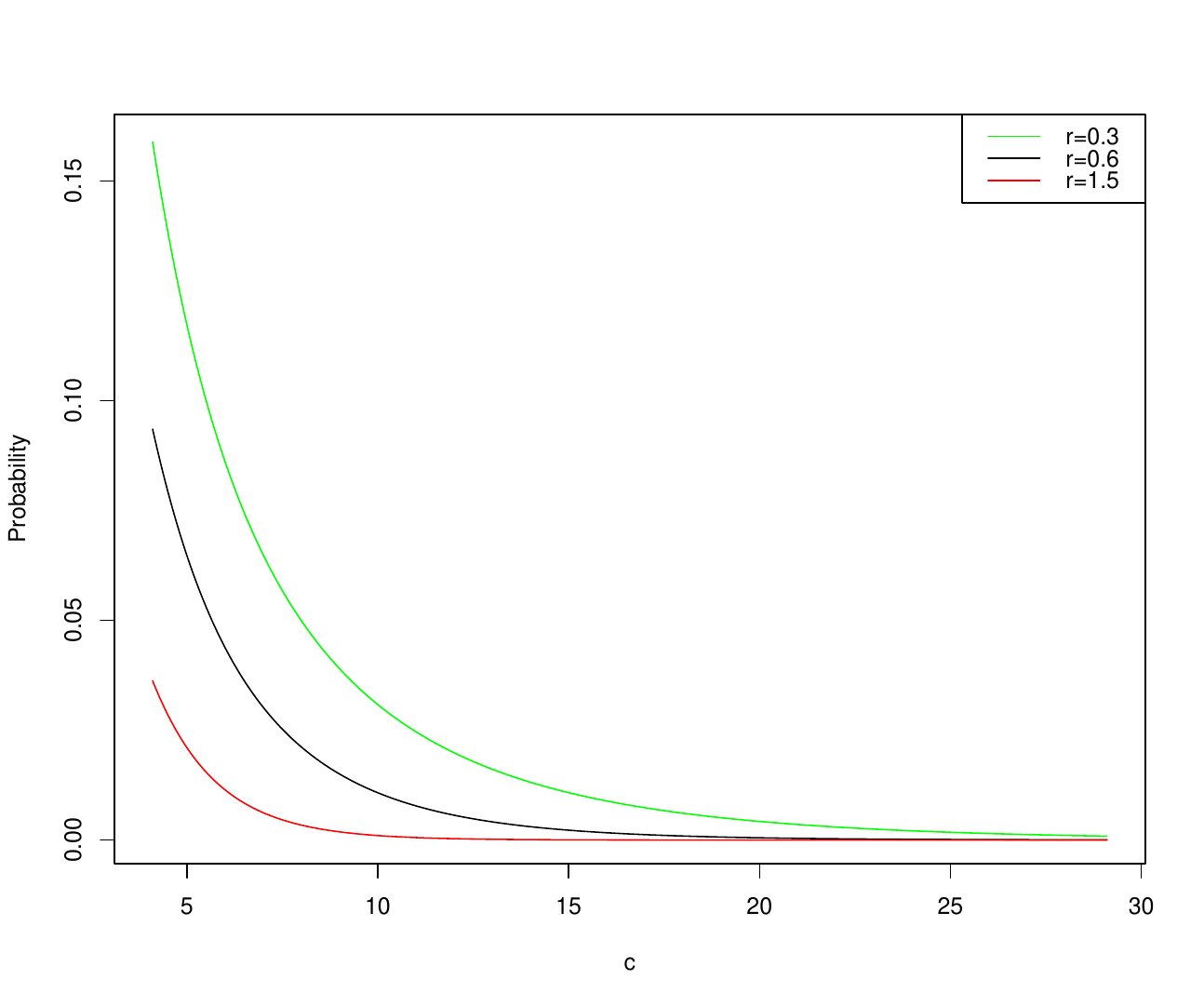}
\caption{Values of  $\mathbb{P}_{x}\left( \underline{U}_{\infty }>r,\tau _{0}^{-}<\infty  \right)$ for the Cramér-Lundberg process with $\alpha=1/2$, $\eta=2$ and $x=1$.}\label{fig:variesvar2}
\end{figure}
%\begin{remark}
%The above result holds also for a process with unbounded variable since, in this case,  $W(0)=0$ and hence $\mathbb{P}_{x}\left( \underline{U}_{\infty }>r \right) =0$.
%\end{remark}
\begin{remark}
From the fact that $\lim_{x \rightarrow \infty } W(x)= \frac{1}{\e[X_1]} $, we obtain
\begin{equation*}
\lim_{x \rightarrow \infty}\mathbb{P}_{x}\left( \underline{U}_{\infty }>r,\tau _{0}^{-}<\infty  \right)=0.
\end{equation*}
\end{remark}
\begin{remark}
It is also possible to derive the distribution of the shortest duration up to the random times $\tau_b^+$ and $T_b^+$ and $\mathrm{e}_{q }$ as in previous subsections.
\end{remark}
The explicit formula in Theorem~\ref{thmmin} allows for a sensitivity analysis of the value of the distribution of shortest negative duration for the model in \eqref{CL} with respect to the premium rate $c$ and the delay parameter $r$. In Figure \ref{fig:variesvar2} and for a fixed initial capital $x$ we observe that for a fixed delay parameter $r$, the probability $\mathbb{P}_{x}\left( \underline{U}_{\infty }>r,\tau _{0}^{-}<\infty \right)$ decreases when the premium rate $c$ increases. This is because negative excursions are shorter when more premiums are injected and thus, the probability decreases when the duration $r$ increases but increases for lower values of $r$.
\subsection{Joint Distribution of $\underline{U}_{\infty}$ and $\bar{U}_{\infty}$}
Now, we are interested in the joint cumulative distribution function of the shortest and longest negative excursion as shown in the 3-dimensional (3D) plot in Figure \ref{fig:variesvar}.
\begin{theorem}
Let $0<u<v$, $x\in \mathbb{R}$ and $\e[X_1]>0$. If $W(0)>0$ (or equivalently $X$ has paths
of bounded variation), we have
\begin{align*}
\p_x(\underline{U}_{\infty}\leq u,\bar{U}_{\infty}\leq v,\tau^-_0 <\infty)
=&\mathbb{E}\left[ X_{1}\right]\left( \frac{\Lambda \left( x,v\right) }{\Upsilon\left( v\right) }-W(x) \right)\\
- &\frac{\mathbb{E}\left[ X_{1}\right] W\left( 0\right)\left(\Lambda \left( x,v\right)-\Lambda \left( x,u\right) -W\left( x\right) (\Upsilon \left( v\right)-\Upsilon \left( u\right)\right)}{1+W\left( 0\right) (\Upsilon \left( v\right)-\Upsilon \left( u\right))}.
\end{align*}
\end{theorem}
\begin{proof}
 First, we have
\begin{align*}
\p_x(\underline{U}_{\infty}>u,\bar{U}_{\infty}\leq v,\tau^-_0 <\infty)=&\sum_{n=1}^{\infty }\mathbb{P}_x\left( N_\infty =n\right) \mathbb{P}_{x}\left(D_{(1)}>u,D_{(n)}\leq v|N_\infty =n\right) \\
=&\sum_{n=1}^{\infty }\mathbb{P}_x\left( N_\infty =n\right)\frac{\mathbb{P}_{x}\left(u<D_1 <v\right)}{\p_x(\tau^-_0 <\infty)}\left(\frac{\mathbb{P}\left(u<D_2 <v\right)}{\p(\tau^-_0 <\infty)} \right)^{n-1}
\\
=&\sum_{n=1}^{\infty }\mathbb{P}_x\left( N_\infty =n\right)\frac{\Lambda \left( x,v\right)-\Lambda \left( x,u\right) -W\left( x\right) (\Upsilon \left( v\right)-\Upsilon \left( u\right))}{\p_x(\tau^-_0 <\infty)}\\
\times&
\left(\frac{ W\left( 0\right) (\Upsilon \left( u\right)-\Upsilon \left( v\right))}{\p(\tau^-_0 <\infty)}\right)^{n-1}\\
=&\frac{\p(\tau^-_0 =\infty)\left(\Lambda \left( x,v\right)-\Lambda \left( x,u\right) -W\left( x\right) (\Upsilon \left( v\right)-\Upsilon \left( u\right)\right)}{1+W\left( 0\right) (\Upsilon \left( v\right)-\Upsilon \left( u\right))},
\end{align*}
and using \eqref{thm11},
\begin{align*}
\p_x(\underline{U}_{\infty}\leq u,\bar{U}_{\infty}\leq v,\tau^-_0 <\infty)=&
\p_x(\bar{U}_{\infty}\leq v,\tau^-_0 <\infty)-\p_x(\underline{U}_{\infty}> u,\bar{U}_{\infty}\leq v,\tau^-_0 <\infty)\\
=&\mathbb{E}\left[ X_{1}\right]\left( \frac{\Lambda \left( x,v\right) }{\Upsilon\left( v\right) }-W(x) \right)\\
- &\frac{\mathbb{E}\left[ X_{1}\right] W\left( 0\right)\left(\Lambda \left( x,v\right)-\Lambda \left( x,u\right) -W\left( x\right) (\Upsilon \left( v\right)-\Upsilon \left( u\right)\right)}{1+W\left( 0\right) (\Upsilon \left( v\right)-\Upsilon \left( u\right))}.
\end{align*}
\end{proof}
\begin{remark}
If $X$ has unbounded variation, the result in Theorem 5 reduces to Eq. \eqref{recover}.
\end{remark}
\begin{figure}[h!]
\includegraphics[scale=0.8,height=8cm]{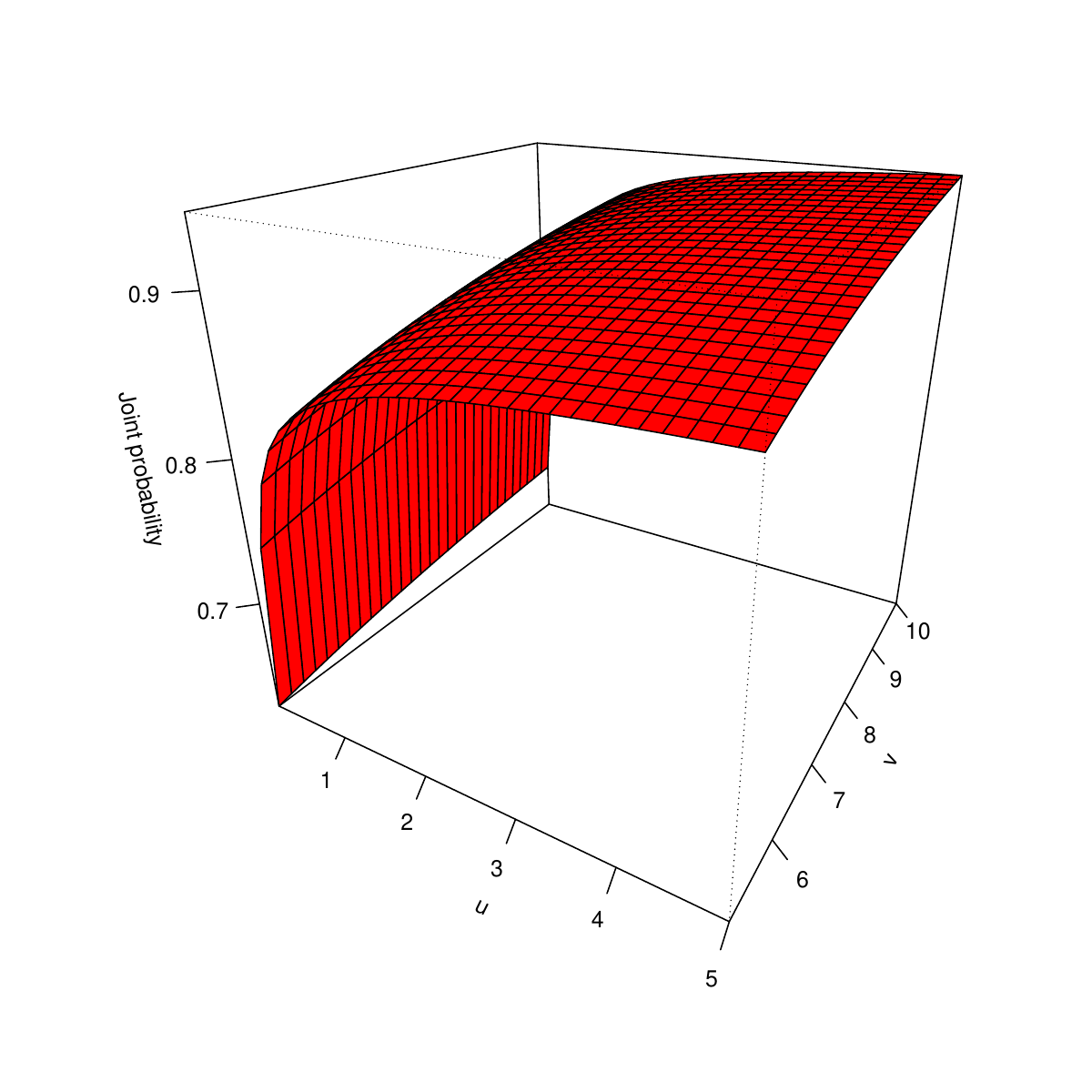}
\caption{Values of the joint cumulative distribution function of $(\underline{U}_{\infty},\bar{U}_{\infty})$ for the Cramér-Lundberg process with $\alpha=1/2$, $\eta=2$, $c=5.5$ and $x=1$.}\label{fig:variesvar}
\end{figure}

\subsection{The range of negative excursions}
For risk processes with bounded variation, the range of negative excursions is defined as the difference between the longest and shortest negative excursions and given by
\begin{equation}
\mathcal{R}_\infty =\bar{U}_\infty  - \underline{U}_\infty = D_{(N)}-D_{(1)},
 \label{supU}
\end{equation}
%where $\underline{U}_t =\inf_{s\leq t} U_s $.
 where $N$ is defined in \eqref{defN}. Hence, making use of the connection between the random variables $\bar{U}_\infty$, $\underline{U}_\infty$ and the order statistics, we are able to derive the following results for the distribution of $\mathcal{R}$. First, observe that
%\begin{equation*}
%\mathbb{P}\left( \mathcal{R}_{\infty} \in \md y \right) =\mathbb{P}\left( \tau
%_{0}^{-}=\infty\right)\delta_0 ( \md y) +\mathbb{P}\left( \mathcal{R}_{\infty} %\in \md y,\tau
%_{0}^{-}<\infty\right)
%\end{equation*}
\begin{equation*}
\p\left(\mathcal{R}_{\infty}  <r \right)=\p\left(\tau_{0}^{-}=\infty \right)+\p\left( \mathcal{R}_{\infty} <r,\tau_{0}^{-}<\infty\right).
\end{equation*}
where
\begin{eqnarray}\label{Range_U}
\mathbb{P}\left( \mathcal{R}_{\infty }<r,\tau _{0}^{-}<\infty \right)
=\sum_{n=1}^{\infty }\mathbb{P}\left( N_\infty  =n\right) \mathbb{P}\left(
D_{\left( n\right) }-D_{\left( 1\right) }<r| N_\infty =n \right)
\end{eqnarray}
and
\begin{eqnarray}\label{range}
&&\mathbb{P}\left( D_{\left( n\right) }-D_{\left( 1\right) }<r| N_\infty =n\right)=n\int_{0}^{\infty }f_{D_{1}}(y)\left( F_{D_{1}}(y+r)-F_{D_{1}}(y)\right)
^{n-1}\mathrm{d}y
\end{eqnarray}
for
$$F_{D_{1}}(y) =\mathbb{P}\left( D_{1}<y|R_{n}<\infty ,L_{n+1}=\infty
\right) =\frac{1 -W(0)\Upsilon \left( y\right) }{1-%
\mathbb{E}\left[ X_{1}\right] W(0)} $$
and $f_{D_{1}}(y)=\frac{\partial }{\partial y} F_{D_{1}}(y)$ is the probability density function of $D_1$. In Figure \ref{fig:variesrange}, the joint cumulative distribution function in \eqref{Range_U} decreases as $c$ increases and $r$ is fixed. This is an intuitive result as when the premium rate increases, the ruin event is less likely to occur and negative excursions are shorter.
%$$
%f_{D_{1}}(y) =\mathbb{P}\left( D_{1}\in \mathrm{d}y|R_{n}<\infty
%%,L_{n+1}=\infty \right) =-\frac{W(0)}{1-\mathbb{E}\left[ X_{1}\right] W(0)} %\frac{\partial }{%
%\partial y}\Upsilon \left( y\right).
%$$

%\begin{theorem}
%For $r>0$ and $x \in \reals$ we have
%\begin{equation} \label{Range_cU}
%\mathbb{P}\left(  \mathcal{R}_{\infty} < r, \tau^-_0 <\infty \right)=
%\sum_{n=1}^{\infty}\mathbb{P}\left( N_\infty =n \right)
%\mathbb{P}\left( D_{\left( n\right) }-D_{\left( 1\right) }<r| N_\infty =n \right)  ,
%\end{equation}
%where $\mathbb{P}\left( N_\infty  =n \right)$ and $\mathbb{P}\left( D_{\left( n\right) }-D_{\left( 1\right) }<r| N_\infty  =n\right) $ are given in \eqref{distN} and \eqref{range} respectively.
%\end{theorem}
\begin{figure}[h]
\includegraphics[scale=1,height=8cm]{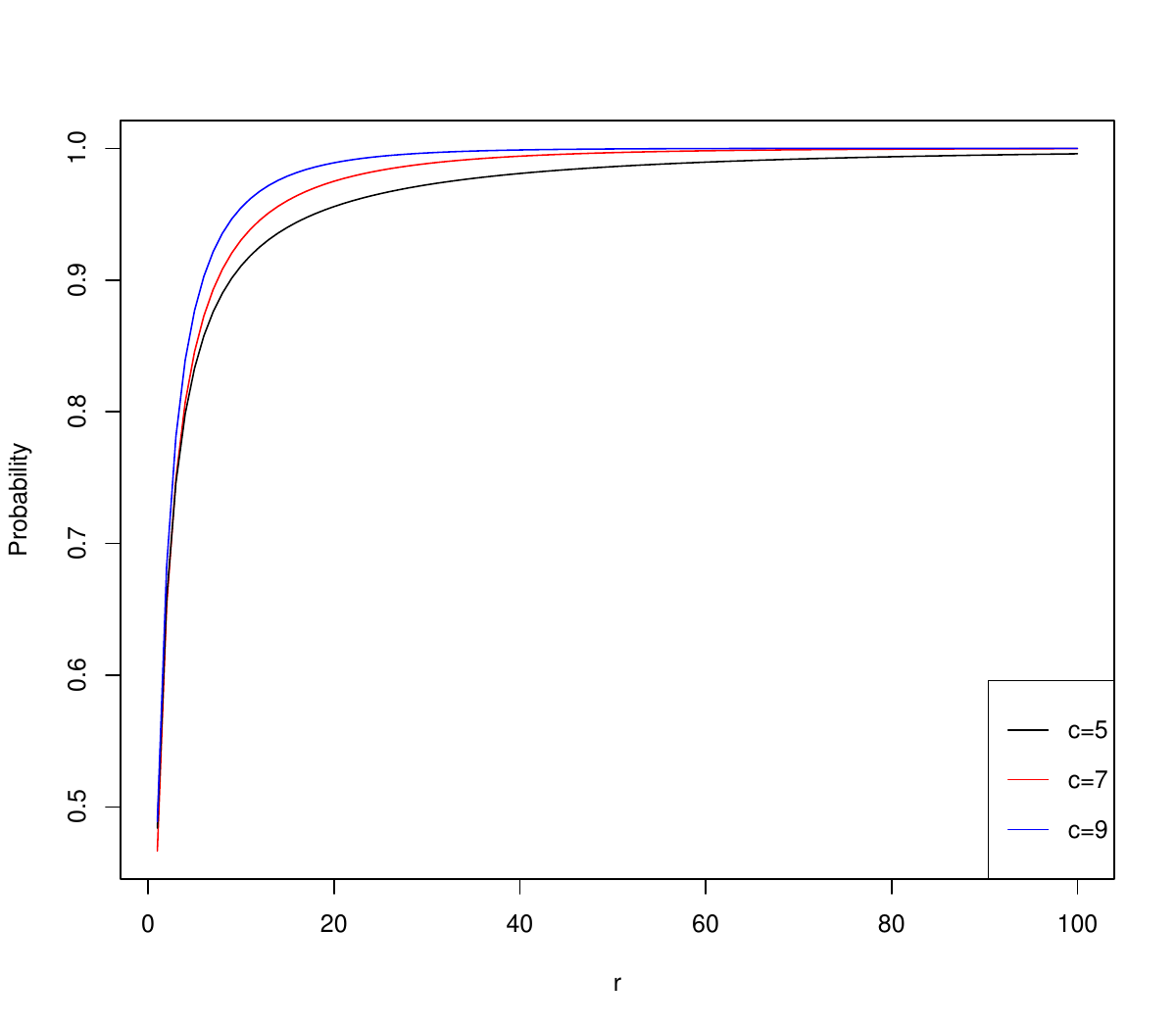}
\caption{Values of  $\mathbb{P}\left(  \mathcal{R}_{\infty} < r, \tau^-_0 <\infty \right)$ for the Cramér-Lundberg process with $\alpha=1/2$, $\eta=2$.}\label{fig:variesrange}
\end{figure}
%\subsection{The drawdown process of $U$}
\section{Applications}\label{Parisian}
This section is devoted to introducing new applications of the longest negative excursion using the binomial approach for processes with bounded variation.
%\subsection{Parisian ruin with fixed delay}
%As pointed out in \cite{loeffenetal2013},  $\int_{0}^{\infty }q\mathrm{e}^{-qr}\mathbb{P}_{x}\left( \kappa _{r}>\mathrm{e}_{\theta }\right) \mathrm{d}r$ does not equal $\mathbb{P}_{x}\left( \kappa _{\me_q}>\mathrm{e}_{\theta }\right) $ since in the definition of $\kappa_{\me_q}$ in \cite{landriaultetal2014} does not have a single underlying exponential random variable $\me_q$, but a whole sequence attached to each negative excursion.
%\begin{remark}
%It is also possible to link the process $U$ to the occupation time on the red of $X$ and modelled by $\mathcal{O}^{X,-}_{t}=\int_{0}^{t}\mathbf{1}_{\left( -\infty ,0\right) }\left(X_{r}\right) \mathrm{d}s,$ which represents the amount of time that a risk process $X$ stays in the negative half-line up to time $t>0$. Indeed, it should be easy to see that $\mathcal{O}^{X,-}_{t}$ can be written as
%$$\mathcal{O}^{U,+}_{t} =\int_{0}^{t}\mathbf{1}_{\left( 0,\infty \right) }\left(U_{s}\right) \mathrm{d}s,$$
%which can be decomposed as $\sum_{n=1}^{N}D_i$.
%Moreover, defining the Parisian ruin with exponential times
%\begin{equation*}
%\kappa^{q} =\sup \left\{ t\geq 0 :\text{ }U_{t}-\me^{g_{t} }_q >0\right\}.
%\end{equation*}
%where $\me^{g_{t}}_q$ denotes an independent copy of $\me_q$ generated for the negative excursion that began at time $g_t$. Thus,  by Proposition 3.3 in \cite{guerinrenaud2015}, we have
%\begin{equation*}
%\p_x\left(\kappa^{q}\leq t\right)=\p_x\left(\tilde{\mathcal{O}}_t > \me_q\right).
%\end{equation*}
%\end{remark}
\subsection{New Parisian ruin models}
We now present two new Parisian ruin models making use of the binomial expansion approach to derive the joint distribution of the total and longest negative excursion.
\subsubsection{Parisian ruins under the same roof}
Under the classical Parisian ruin model, a clock is assigned for each negative excursion. In \cite{guerin_renaud_2015}, cumulative ruin has been introduced in which the cumulative sum of excursions below zero is considered and compared to a fixed delay. In this section, we are interested in a new type of Parisian ruin in which both ruin models are considered. Thus, for $r<l$, we define
$$
\rho_{r,l} = \min( \sigma_l , \kappa_r),$$
where $\sigma_l = \inf \left\lbrace t > 0 \colon \mathcal{O}^{X}_{t} > l \right\rbrace$ and $\mathcal{O}^{X}_{t} =\int_{0}^{t}\mathbf{1}_{\left( -\infty,0 \right) }\left(X_{s}\right) \mathrm{d}s$ is the occupation time in the red of $X$ which represents the amount of time that $X$ stays in the negative half-line. %Hence, in this case, ruin occurs if the length of single (cumulative) negative excursions is greater than $r>0$ ($l>0$).
Clearly, for fixed values $t$, we have
$$
\left\{ \rho_{r,l}  \geq t  \right\} = \left\{ \mathcal{O}^{X}_{t} \leq l , \Bar{U}_t \leq r \right\},$$
%so,
%$$\p( \rho_{r,l}  \geq t ) = \p( \mathcal{O}^{X}_{t} \leq l, \Bar{U}_t \leq r ).$$
and consequently
$$\p( \rho_{r,l} = \infty  ) = \p(  \mathcal{O}^{X}_{\infty } \leq l ,  \Bar{U}_\infty \leq r ).$$
As the above probability is closely related to the joint distribution of $\mathcal{O}^{X}_{\infty } $ and $ \bar{U}_\infty $, we will now study the latter for bounded variation surplus process $X$. First, we know that $\mathcal{O}^{X}_{ \infty }$ can be decomposed as
$$\mathcal{O}^{X}_{\infty}=\sum_{n=1}^{N_\infty}D_n .$$ Thus, we need the joint distribution of the sum and maximum of $D_1, D_2, ... ,D_{N_\infty} $. \\ We have
\begin{align*}
F_N (l,r)=\p ( \mathcal{O}^{U}_{\infty} \leq l,\bar{U}_\infty \leq r)
=\p(\tau^-_0 = \infty)+ \sum_{n=1}^{\infty }  \p(N_\infty =n) F_n (l,r),
\end{align*}
where $F_n (l,r) =\p (\sum_{i=1}^{n}  D_i \leq l,D_{(n)} \leq r)$ is the joint
cumulative distribution function with initial condition $F_1 (l,r)=F_{D_{1}}(\min(l,r) ) $. Using Eq. $(2.8)$ in \cite{qeadan2012joint} (see also Theorem $2.7$ in \cite{efrem2023recurrence}), we have the following recursive relation of the joint pdf
\begin{align}\label{recursivepdf}
f_n (l,r) =  \frac{\partial^2 }{\partial l \partial r} F_n (l,r)  = n f_{D_1} (r) \int^{\min(l-r,r ) }_{\frac{l-r}{n-1}} f_{n-1} (l-r,t) \md t ,   &
\end{align}
for all $(l,r)\in   \left\{ (l,r)\in \reals_+^2  : \frac{l}{n}\leq r\leq l \right\} $ with initial condition
$$f_0 (l,r) =f_{D_1} (l) \delta_0 (l-r) =  f_{D_1} (r) \delta_0 (l-r). $$
\subsubsection{Parisian ruin under a peak-to-sum constraint}
In the following type of Parisian ruin, we consider the peak-to-sum ratio of negative excursions, that is,
$$S_t =\frac{\Bar{U}_t}{ \mathcal{O}^{X}_{t}} \text{   on } \tau^-_0 <\infty , $$
which represents the proportion of the longest duration wrt the total duration of all distress events. Thus, for $\alpha,r>0$, we define the following ruin time
$$
\zeta_{\alpha, r }= \inf \left\lbrace t > 0 \colon S_t > \alpha \text{ and } U_t >r  \right\rbrace .
$$
Under the above Parisian ruin time, we monitor negative excursions with length greater than $r$ and compare the peak-to-sum $S_t$ to a fixed parameter $\alpha >0$. \\
For fixed $t$, we obtain
$$\p(\zeta_{\alpha, r}  \leq t ) = \p\left( S_{t} > \alpha, \Bar{U}_t >r \right),$$
and
$$\p( \zeta_{\alpha, r }< \infty  ) =\p\left( S_{\infty} > \alpha ,  \Bar{U}_\infty >r \right)=  \p( \Bar{U}_\infty >  \alpha \mathcal{O}^{X}_{\infty}  , \Bar{U}_\infty > r ).$$
In order to compute the above probability, we define the peak-to-sum ratio as
$$
R_n =  \frac{D_{(n)}}{\sum_{i=1}^{n} D_i}. $$
Then,
$$\p( R_n >\alpha , D_{(n)} >r )=\int^{ \infty }_{r}  \int^{y/\alpha }_{0} f_n (y,z) \md y \md z,     $$
where $f_n $ is given in \eqref{recursivepdf} and consequently
$$ \p( \zeta_{\alpha, r} < \infty  ) =  \sum_{n=1}^{\infty }  \p( R_n > \alpha , D_{(n)}>r )  \p(N_\infty =n). $$
\subsection{Stochastic ordering of Parisian ruin probabilities}
We recall that $X$ is smaller than $Y$ in the stochastic dominance order, which is denoted by $X\preceq _{st}\tilde{X}$, if
\begin{equation}\label{str1}
\bar{F}_{X}(u) \leq \bar{F}_{\tilde{X}}(u) , \; \text{for all $u$,}
\end{equation}
where $\bar{F}_{X}(u)=1-F_{X}(u)$.
In particular, if $N$ and $\tilde{N}$ are discrete random variables taking on values in $\mathbb{N}$, then $N\preceq _{st}\tilde{N}$ if and only if
\begin{equation}\label{discorder}
\sum_{i=0}^{n}\mathbb{P}\left(N  =i\right) \geq \sum_{i=0}^{n}\mathbb{P}\left( \tilde{N}=i\right) .
\end{equation}
Let $\varphi _1$ and $\varphi _1$  be two increasing functions that satisfy $\varphi _1 (x) \leq \varphi _2 (x)$ for all $x \in \reals$. Then $$\varphi _1 (X) \preceq _{st} \varphi _2 (X) ,$$
and if $X \preceq _{st} Y $,
\begin{equation}\label{str2}
\varphi _1 (X) \preceq _{st} \varphi _2 (Y).\end{equation}
Finally, if $X=\left\lbrace X_t , t \geq 0 \right\rbrace$ and $\tilde{X}=\left\lbrace \tilde{X}_t , t \geq 0 \right\rbrace$ are stochastic processes, then we write $X\preceq_{st}\tilde{X}$ if, for each $t \geq 0$, we have
$$
X_{t}\preceq_{st}\tilde{X}_{t}.
$$

%We define the aggregate loss at time $t$ by $M_t=A_t-ct$ (see \eqref{CL} for definition of $A$), thus the ruin probabilities is
%\begin{equation}\label{str}
%\p_x(\tau^-_0 <\infty)=\p_x(\Bar{M}_\infty >x),
%\end{equation}
%where $\Bar{M}_\infty =\underset{t\geq 0}{\max} \;M_t$ is the maximal aggregate loss of the process.
%Now, making use of the relationship in \eqref{SNLPPr2}, we will establish a stochastic ordering of Parisian ruin probabilities. First, let us give some background material %on stochastic dominance. Consider two random variables $X$ and $\tilde{X}$, and let $\bar{F}_{X}$ and $\bar{F}_{\tilde{X}}$ be their survival functions.
We consider
%let $M$ and $\tilde{M}$ be two aggregate loss amount processes associated with two aggregate claim amount processes $A$ and $\tilde{A}$, themselves from
two Cram\'er-Lundberg risk processes $X$ and $\tilde{X}$ as defined in~\eqref{CL} with the same safety loading. For any stopping time  $\tau$ for $X$ by $\tilde{\tau}$ we will denote its counterpart for $\tilde{X}$.
%be the passage times associated with $\tilde{X}$ and $\tilde{N}$ the analogue of $N$ for $\tilde{X}$.
%If $M \preceq_{st} \tilde{M}$, using \eqref{str} and \eqref{str1}, we have
%Let $X = \{X_t, t\ge 0\}$ be a Cramér-Lundberg risk model defined in \eqref{CL}.

Under the stochastic dominance ordering
 of the claim sizes
 \begin{equation}\label{STclaimsizes}
 C_i \preceq _{st}\tilde{C}_i,\qquad i=1,2,\ldots
 \end{equation}
 the comparison of ruin probabilities has been solved using the Pollaczeck–Khinchine formula
 %which states that the probability of ruin is equal to the tail distribution function of a compound geometric random variable and used to order ruin probabilities
(see Theorem $5.4.4$ in \cite{MR16802} for details) giving
\begin{equation}\label{orderuin}
\p_{x}\left(\tau^-_0 <\infty \right)  \leq \p_{x}\left(\tilde{\tau}^-_0<\infty \right).
\end{equation}
The comparison of finite-time ruin probabilities has also been proved in \cite{lefevre2017some}.

We will now prove the following result.
\begin{prop}\label{properties}
For $r>0$ and $x\in \reals$, if \eqref{STclaimsizes} holds true,
%$M \preceq_{st} \tilde{M}$,
then $\bar{U}_{\infty } \preceq_{st} \bar{\tilde{U}}_{\infty }$, and consequently
\begin{equation}\label{monotonicity}
\p_{x}\left(\kappa_{r}<\infty \right) \leq \p_{x}\left(\tilde{ \kappa}_{r}<\infty \right) .
\end{equation}
\end{prop}
\begin{proof}
First, we need to prove that $N_\infty \preceq_{st} \tilde{N}_\infty $. Using \eqref{distN}, we obtain
\begin{eqnarray*}
\sum_{i=0}^{n}\mathbb{P}_{x}\left( N_\infty =i\right)  &=&\mathbb{P}_{x}\left( \tau
_{0}^{-}=\infty \right) +\mathbb{P}_{x}\left( \tau _{0}^{-}<\infty \right)
\mathbb{P}\left( \tau _{0}^{-}=\infty \right) \sum_{i=1}^{n} \mathbb{P}%
\left( \tau _{0}^{-}<\infty \right)  ^{i-1}  \\
&=&1-\mathbb{P}_{x}\left( \tau _{0}^{-}<\infty \right) \mathbb{P}\left( \tau
_{0}^{-}<\infty \right) ^{n}.
\end{eqnarray*}
Hence, using \eqref{orderuin}, we obtain
\begin{equation}\label{orderN}
\sum_{i=0}^{n}\mathbb{P}_{x}\left(N_\infty  =i\right) \geq \sum_{i=0}^{n}\mathbb{P}%
_{x}\left( \tilde{N}_\infty =i\right),
\end{equation}
and $N_\infty \preceq_{st} \tilde{N}_\infty $ follows from \eqref{discorder}. \\
For $y\leq 0$, we have $\p_{y}\left(\tau^+_0 <r\right)  \geq \p_{y}\left(\tilde{\tau}^+_0<r \right),
$ as $  \tau^+_0 \preceq_{st} \tilde{\tau} ^+_0   $, see the discussion below Theorem 6.B.34 in \cite{shaked_Shanthikumar2007}. Moreover,
using Theorem $1$ in \cite{MR1664021} together with \eqref{orderuin} and \eqref{STclaimsizes}, we have $\Tilde{X} _{\Tilde{ \tau}^-_0}   \preceq_{st}  X_{\tau^-_0}   $
and consequently $$ \e\left[ \p_{X_{\tau^-_0}}\left(\hat{\tau}^+_0 <r\right)\right]  \geq
\e \left[ \p_{_{\tilde{X} _{\Tilde{ \tau}^-_0}}}\left(\tilde{\hat{\tau}}^+_0<r \right)\right] , $$
which follows from \eqref{str2}. Combining the above inequality with \eqref{D1} and \eqref{orderuin}, we deduce
\begin{equation*}
\mathbb{P}_{x}\left(D_{i}<r  \right) \geq \mathbb{P}_{x}\left( \tilde{D}_{i}<r \right),
\end{equation*}
and thus $D_{i} \preceq_{st} \tilde{D}_i $ for $i=1,...,n$. Consequently, using Theorem $4.1$ in \cite{balakrishnan1998handbook}, we have
\begin{equation}\label{orderstat}
 D_{(n)} \preceq_{st} \tilde{D}_{(n)},
 \end{equation}
and combined  $N_\infty \preceq_{st} \tilde{N}_\infty $, we have
\begin{equation*}
\p_{x}\left( D_{(N_\infty)} <r\right)  \geq \p_{x}\left(\tilde{D}_{(\Tilde{N}_\infty)} <r \right),
\end{equation*}
and the result follows from \eqref{SNLPPr2}.
\end{proof}
\subsection{On the number of near-maximum distress periods}
A near-maximum distress period is the negative excursion with length falling within a fixed duration $a$ of the current largest negative excursion. Assuming that $X$ has paths of bounded variation, for a fixed $a>0$, we define
$$
\mathcal{E}_a=\#\left\{1 \leq j \leq N_\infty : D_{j} \in\left(\bar{U}_\infty -a, \bar{U}_\infty \right]\right\},
$$
as the number of near-maximum distress periods. The above quantity is related to the work of Li and Pakes \cite{li2001number} in which the number of near-maximum insurance claims has been studied.
%It provides information on how often a negative excursion will be close to the existing record claim during $[0, t]$ ",
%and $\mathcal{E}_a \mathcal{M}_{t}$ estimates the total value of these claims. It is explained in Section 8.7 of Embrechts et al. (1997) how this is related to problems in reinsurance. Hashorva and Hüsler (1999) extend the notion of near-maxima by fixing $b<a$ and considering the counts $K_{n}(a)-K_{n}(b)$. Their limit results can be extended to our random sample context, thus giving results about the number of claim sizes falling in a band which lies below the current record claim.
%We remark that, for practical purpose, it often suffices to consider processes $\{N_t\}$ satisfying
%$$
%\frac{N_t}{t} \stackrel{p}{\rightarrow} \lambda
%$$
%where $\lambda^{-1}$ is the mean interarrival time of claims, assumed finite. This includes the renewal counting process, including the homogeneous Poisson process (the Cramér-Lundberg risk model) — see Section 2.5.2 in Embrechts et al. (1997).\\ The definitions (1.1) and (1.2) yield the fundamental representation
%$$
%\mathcal{E}_a=\mathcal{E}_{N}(a)=\#\left\{1 \leq j \leq N_t: D_{j} \in\left(M_{N}-a, M_{N}\right]\right\} .
%$$
Using Eq. (2.5) in \cite{pakes1997number}, we have the following conditional p.g.f.
$$
\e \left[s^{\mathcal{E}_a} \mid N_\infty \right]=s N_\infty \int_{0}^{\infty}\{F_{D_ {1}} (x-a)+s[F_{D_ {1}}(x)-F_{D_ {1}}(x-a)]\}^{N_\infty-1} f_{D_ {1}}(x) \md x ,
$$
where $f_{D_ {1}}$ is the probability density function of $D_1$. Let $h(s)=\e\left[s^{N_\infty }\right]$ and $h^{\prime}(s)=\e\left[N_\infty s^{N_\infty -1}\right]$ for $0 \leq s<1$. The p.g.f. of $\mathcal{E}_a$ is then given by
$$
\e\left[s^{\mathcal{E}_a}\right]=s \int_{0}^{\infty} h^{\prime}\left( F_ {D_ {1}} (x-a)+s[F_{D_ {1}}(x)-F_{D_ {1}}(x-a)]\right)  f_{D_ {1}}(x) \md x ,
$$
and finally
$$
 \e\left[\mathcal{E}_a\right]=\int_{0}^{\infty} h^{\prime}\left( F_{D_ {1}}(x+a)\right) f_{D_ {1}}(x) \md x ,
$$
that is the expected number of near-maximum distress periods.\\
\begin{figure}[!h]
	\centering
	\subfloat{{\includegraphics[width=7cm,height=7cm]{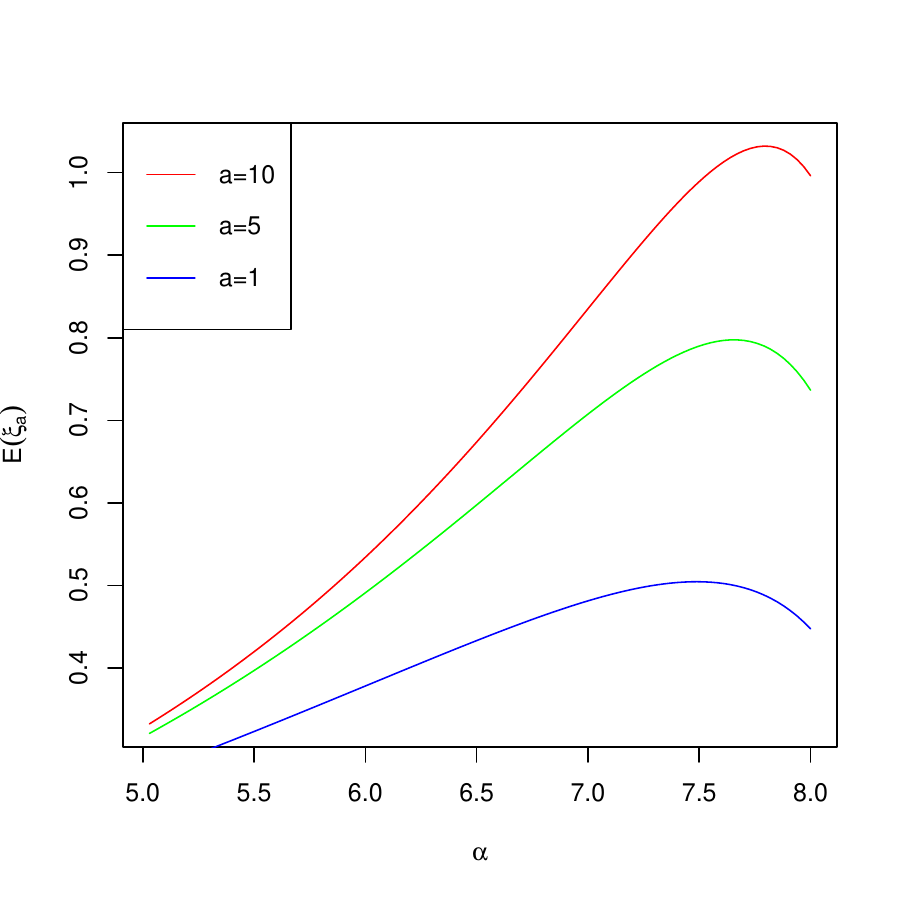} }}
	\qquad
	\subfloat{{\includegraphics[width=7cm,height=7cm]{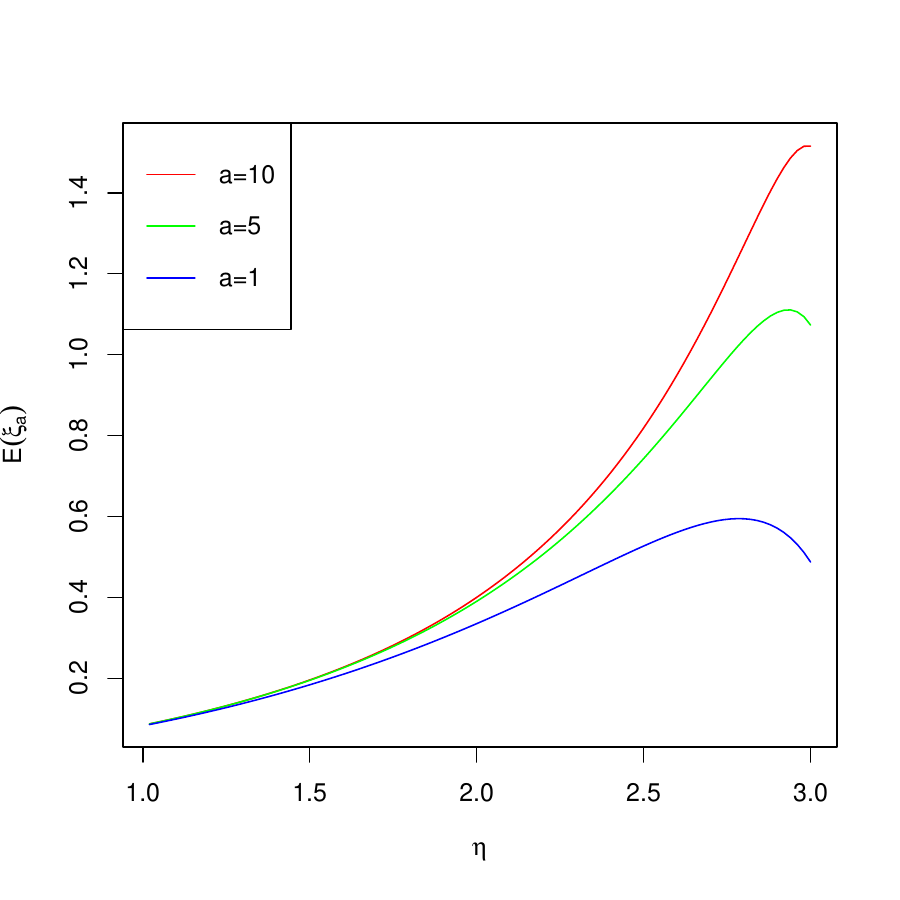} }}
 \qquad
	\subfloat{{\includegraphics[width=7cm,height=7cm]{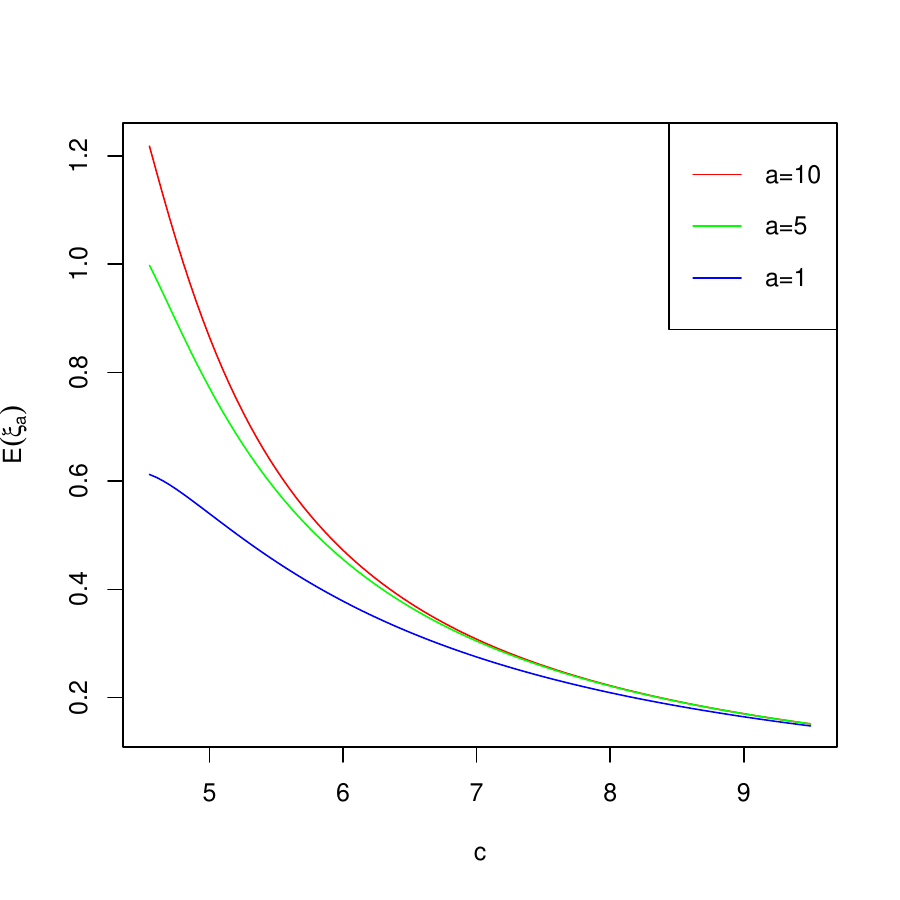} }}
	\caption{Values of  $ \e\left[\mathcal{E}_a\right]$ for the Cramér-Lundberg process with ($c=8.5$, $\eta=1$), ($c=9.5$, $\alpha=1/3$) and ($\eta=2$, $\alpha=1/2$).  }
	\label{figxi}
\end{figure}
We can now calculate the above expression for the Cram\'{e}r-Lundberg process in \eqref{CL}. In Figure \ref{figxi}, the values of $ \e\left[\mathcal{E}_a\right]$ increase as $\eta$ increases, then decrease for large values of $\eta$. This reflects the increase of the number of negative excursions when the size of the claims increases while it decreases as the process struggles to recover above $0$. The claim parameter $\alpha$ has also a similar impact on $ \e\left[\mathcal{E}_a\right]$. In the third plot, when the premium rate increases, the ruin event is less likely to occur and negative excursions are shorter which decreases the values of $ \e\left[\mathcal{E}_a\right]$. In all three plots, the values of $ \e\left[\mathcal{E}_a\right]$ increase as $a$ increases.
\section{Acknowledgements}
 M. A. Lkabous acknowledges the support from a start-up grant from the University of Southampton. The research of Zbigniew Palmowski is partially supported by Polish National Science
Centre Grant No. 2021/41/B/HS4/00599.
\bibliographystyle{alpha}
\bibliography{REFERENCES}
\end{document}